\documentclass[10pt,a4paper]{article}
\usepackage{graphics,graphicx}
\usepackage{amssymb,amsmath}
\usepackage{enumitem}

\usepackage[capitalize]{cleveref}

\usepackage{color}

\newtheorem{thm}{Theorem}

\newtheorem{problem}{Problem}

\usepackage{thmtools}
\newlist{lemmalist}{enumerate}{1}
\setlist[lemmalist]{label=(\roman{lemmalisti}),
	ref=\thelemma:$(\roman{lemmalisti})$,
	noitemsep}
\declaretheorem[
name=Lemma]{lemma}

\Crefname{lemmalist}{Lemma}{Lemmas}
\addtotheorempostheadhook[lemma]{\crefalias{lemmalisti}{lemmalist}}

\newtheorem{claim}{Claim}[thm]
\newtheorem{claim1}{}[lemma]
\Crefname{claim}{Claim}{Claims}
\addtotheorempostheadhook[claim]{\crefalias{claimi}{claim}}

\def \no {\noindent}
 
 \def \sm {\setminus}
 \def \es {\emptyset}

\newenvironment{proof}[1][]%
{\noindent {\setcounter{equation}{0}\it Proof.
}{#1}{}}{\hfill$\Box$\vspace{0ex}}

\newenvironment{proof2}[1][]%
{\noindent{\setcounter{equation}{0}\it Proof.
}{#1}{}}{$\diamond$\vspace{0ex}}

\newcommand{\res}{P_2+P_3}

\begin{document}
\title{Optimal chromatic bound for  ($P_2+P_3$,  $\overline{P_2+ P_3}$)-free graphs}

\author{ Arnab Char\thanks{Computer Science Unit, Indian Statistical
Institute, Chennai Centre, Chennai 600029, India. } \and T.~Karthick\thanks{Corresponding author, Computer Science Unit, Indian Statistical
Institute, Chennai Centre, Chennai 600029, India. Email: karthick@isichennai.res.in }}

\date{\today}

\maketitle

\begin{abstract}For a graph $G$, let $\chi(G)$ ($\omega(G)$)
denote its chromatic (clique) number.  A $P_2+P_3$ is the graph obtained by taking the disjoint union of a two-vertex path $P_2$ and a three-vertex path $P_3$. A $\overline{P_2+P_3}$ is the complement graph of  a $P_2+P_3$.
In this paper, we study the class of  ($P_2+P_3$, $\overline{P_2+P_3}$)-free graphs and show that every such graph $G$ with
$\omega(G)\geq 3$ satisfies  $\chi(G)\leq \max \{\omega(G)+3, \lfloor\frac{3}{2} \omega(G) \rfloor-1 \}$.  Moreover, the bound is tight. Indeed, for any $k\in {\mathbb N}$ and $k\geq 3$, there  is a ($P_2+P_3$, $\overline{P_2+P_3}$)-free graph $G$  such that $\omega(G)=k$ and $\chi(G)=\max\{k+3, \lfloor\frac{3}{2} k \rfloor-1 \}$.
\end{abstract}

\section{Introduction}
All our graphs are simple, finite and undirected. For general graph theoretic notation and terminology which are not defined here, we follow  West~\cite{West}.
 As usual, let $P_\ell$, $C_{\ell}$ and $K_\ell$ respectively denote the chordless path, chordless cycle and the complete graph on $\ell$ vertices. For a graph $G$, let $\overline{G}$ denote the complement graph of $G$. Given a class of graphs $\cal C$,  $\overline{\cal C}$ is the class of graphs $\{\overline{G}\mid G\in \cal C\}$.   A graph class ${\cal C}$ is \emph{hereditary} if $G\in {\cal C}$, and if $H$ is an induced subgraph of $G$, then $H\in {\cal C}$. Let $G$ be graph with vertex-set $V(G)$ and edge-set $E(G)$.  For $X\subseteq V(G)$, we write $G[X]$ and $G-X$ to denote the subgraph of $G$ induced on $X$ and $V(G)\sm X$ respectively. We say that $X$ induces a graph $H$ if $G[X]$ is isomorphic to $H$.
 For any two disjoint subsets $X$ and $Y$ of $V(G)$, we say that $X$ is
\emph{complete}   to $Y$  if every vertex in $X$ is adjacent   to every vertex in $Y$; and $X$ is \emph{anticomplete} to $Y$ if there is no edge in $G$ with one end in $X$ and the other in $Y.$
For two vertex disjoint  graphs $G$ and $H$,   $G+H$ is the disjoint union of $G$ and $H$, and $G\vee H$
is the graph obtained from $G+H$ and such that $V(G)$ is complete to $V(H)$.
 We say that a graph $G$ \emph{contains} a graph $H$ if $G$ has an induced subgraph which is isomorphic to $H$.
A graph is ($H_1, H_2,\ldots, H_k$)-free  if it does not contain any graph in $\{H_1, H_2, \ldots, H_k\}$.
Note that a graph  $G$ is ($H_1, H_2,\ldots, H_k$)-free if and only if   $\overline{G}$ is $(\overline{H_1}, \overline{H_2},\ldots, \overline{H_k})$-free.

In a graph $G$, a \emph{clique}  is a set of mutually adjacent  vertices in $G$, and a  \emph{stable set} is a set of mutually nonadjacent  vertices in $G$.
As usual, for a graph $G$, we write $\chi(G)$  to denote the  chromatic number  of $G$,  $\omega(G)$ to denote the clique number of $G$ (the size of a maximum clique in $G$),  $\theta(G)$ to denote the clique covering number of $G$ (the minimum number of (disjoint) cliques needed to cover the vertices of $G$), and $\alpha(G)$ to denote the stability number  of $G$ (the size of a maximum stable set in $G$).
 Clearly, for any graph $G$,  $\chi(G) =\theta(\overline{G})$ and $\omega(G)=\alpha(\overline{G})$. A graph $G$ is \emph{perfect} if every induced subgraph $H$ of $G$ satisfies $\chi(H) =\omega(H)$ or $\theta(\overline{H})=\alpha(\overline{H})$.

  In the following, by a function $f$, we mean a function $f:{\mathbb{N}}\rightarrow {\mathbb{N}}$ with $f(1)=1$ and $f(x)\geq x$ for all $x\in {\mathbb{N}}$, where $\mathbb{N}$ is the set of natural numbers. A hereditary class of graphs $\cal C$  is \emph{$\chi$-bounded}, if there is a function $f$  such that every $G\in {\cal C}$ satisfies $\chi(G)\leq f(\omega(G))$; here $f$ is  called a \emph{$\chi$-binding function} for $\cal C$. We say that a function $f$ is the \emph{optimal} $\chi$-binding function for a $\chi$-bounded class of graphs $\mathcal{C}$ if for each $\ell \in \mathbb{N}$, there is a graph $G\in \mathcal{C}$ such that $\omega(G)=\ell$ and $\chi(G)=f(\ell)$.   Likewise, a hereditary class of graphs $\cal C$  is \emph{$\theta$-bounded}, if there is a function $f$  (called a {\it $\theta$-binding function}) such that every $G\in {\cal C}$ satisfies  $\theta(G)\leq f(\alpha(G))$.  Clearly, a function $f$ is a $\chi$-binding function for $\cal C$ if and only if $f$ is a $\theta$-binding function
for $\overline{\cal C}$, and $\cal C$ is $\chi$-bounded if and only if $\overline{\cal C}$ is $\theta$-bounded.

In this paper, we are interested in some self-complementary class of $\chi$-bounded graphs.
A class of graphs $\cal C$ is said to be {\it self-complementary class} if $\cal C = \overline{\cal C}$.   A self-complementary hereditary class of graphs $\cal C$ is $\chi$-bounded  if and only if $\cal C$ is $\theta$-bounded. In particular, if $\cal C$ is $\chi$-bounded, then the optimal $\chi$-binding function of $\cal C$ is  same as the optimal $\theta$-binding function of $\cal C$. For instance, by a result of Lov\'asz \cite{LL}, the class of perfect graphs is a self-complementary class of $\chi$-bounded graphs with $f(x)=x$ as the optimal $\chi$-binding function.

 The notions of $\chi$-bounded and $\theta$-bounded   classes of graphs were introduced in a seminal work of Gy\'arf\'as~\cite{Gy87}, and they have received   a wide attention since then; see~\cite{RS-survey,ScottSey-Survey}.  Among other conjectures and problems, Gy\'arf\'as~\cite{Gy87}   proposed the following.

 \begin{problem}[\cite{Gy87}]\label{mainprob}
For a fixed forest $F$, assuming that the class of  ($F, \overline{F}$)-free graphs $\cal F$ is $\chi$-bounded, what is the optimal $\chi$-binding function for $\cal F$?
\end{problem}

 \begin{figure}[t]
\centering
 \includegraphics[height=1.5cm,width=7cm]{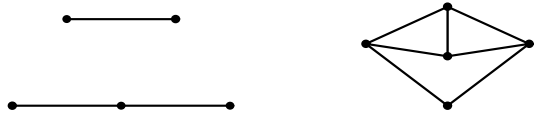}
\caption{A  $P_2+ P_3$  and a $\overline{P_2+ P_3}$ (left to right).}\label{fig:p2p3}
\end{figure}

\cref{mainprob} is open and seems to be hard even when $F$ is  a simple type of forest.
 So it is interesting to look at some special cases, in particular, when $F$ is a   forest on at most five vertices.
 While the optimal $\chi$-binding functions for  classes of ($F, \overline{F}$)-free graphs are known when $F$ is a forest on at most four vertices, except when $F = \overline{K_4}$ \cite{Gy87}, only three classes of graphs  were studied for \cref{mainprob} when $F$ is a five-vertex forest.  Fouquet et al. \cite{FGMT95} showed  that every ($P_5, \overline{P_5}$)-free graph $G$ satisfies $\chi(G)\leq \binom{\omega(G)+1}{2}$, and that there are ($P_5, \overline{P_5}$)-free graphs $G$ with $\chi(G)\geq  \omega(G)^k$, where $k= \log_25- 1$. The second author with Maffray \cite{KMa182} showed that every ($P_4+P_1, \overline{P_4+P_1}$)-free graph $G$ satisfies $\chi(G)\leq \lceil\frac{5\omega(G)}{4}\rceil$, and that the bound is tight.  Recently, Chudnovsky et al. \cite{CCS} showed that every  (fork, anti-fork)-free graph $G$ satisfies $\chi(G)\leq 2\omega(G)$, and that the bound is asymptotically tight. Thus, \cref{mainprob} is open for seven  pairwise nonisomorphic  forests on five vertices, as there are  ten such forests. In this paper, we study the class of  ($P_2+ P_3$, $\overline{P_2+ P_3}$)-free graphs; see Figure~\ref{fig:p2p3}. Randerath et al.~\cite{RST02} showed that every ($P_2+P_3$,  $\overline{P_2+ P_3}$)-free graph $G$ with $\omega(G)=2$ satisfies $\chi(G)\leq 4$, and the well-known Myceilski's $4$-chromatic graph shows that the bound is tight. However, no optimal $\chi$-binding function is known for the class of ($P_2+ P_3$, $\overline{P_2+ P_3}$)-free graphs in general.
  Here, we prove the following.

\begin{thm}\label{thm:p2p3-bnd}
 Every ($P_2+ P_3$, $\overline{P_2+ P_3}$)-free graph $G$ with $\omega(G)\geq 3$ satisfies  $\chi(G)\leq \max \{\omega(G)+3,$ $  \lfloor\frac{3\omega(G)}{2}\rfloor-1 \}.$
\end{thm}

\begin{thm}\label{thm:p2p3-bndopt}
 For every $k\in \mathbb{N}$ and $k\geq 3$, there is a graph $(P_2+P_3,\overline{P_2+P_3})$-free $G$ such that $\omega(G)=k$ and $\chi(G)=\max\{k+3,\lfloor \frac{3}{2}k\rfloor-1\}$.
\end{thm}

Thus,  the function $g:{\mathbb{N}}\rightarrow {\mathbb{N}}$ defined by $g(1)=1$, $g(2)=4$, and $g(x)= \max\{x+3,\lfloor \frac{3}{2}x\rfloor-1\}$, for  $x\geq 3$, is the optimal $\chi$-binding (or $\theta$-binding) function for the class of ($P_2+P_3$,  $\overline{P_2+ P_3}$)-free graphs.
We give a proof of  \cref{thm:p2p3-bnd} in the final section.  To prove  \cref{thm:p2p3-bndopt},    consider the following ($P_2+ P_3$, $\overline{P_2+ P_3}$)-free graphs, where $t\in \mathbb{N}$:

\vspace{-0.25cm}
\begin{enumerate}[leftmargin=1.2cm, label=(\Roman*),series=edu*]\itemsep=0pt
\item Let $Q$ be the  16-regular {\it Schl\"afli graph} on  27 vertices. Then:

\vspace{-0.25cm}
\begin{itemize}\itemsep=0pt
\item $\chi(Q)=9$ and $\omega(Q)=6$.
\item $\chi(\overline{Q})=6$ and $\omega(\overline{Q})=3$.
\item $\chi(Q\vee K_t)=t+9$ and $\omega(Q\vee K_t)=t+6$.
\item $\chi(\overline{Q}\vee K_t)=t+6$ and $\omega(\overline{Q}\vee K_t)=t+3$.
\end{itemize}

\item Let $H$ be the complement of the \emph{Clebsch graph}   on 16 vertices.  Then:

\vspace{-0.25cm}
\begin{itemize}\itemsep=0pt
\item $\chi(H)=8$ and $\omega(H)=5$.
\item  If $G = H-v$, for any $v\in V(H)$, then   $\chi(G)=8$ and $\omega(G)=5$.
\item  $\chi(H\vee K_t)=t+8$ and $\omega(H\vee K_t)=t+5$.
\end{itemize}
\end{enumerate}
We refer to \cite{CS-schlafli} for a precise definition  of the   Schl\"afli graph and  its properties. It is interesting to note that the set of neighbours of any vertex in the 16-regular Schl\"afli graph on  27 vertices induces the complement of the Clebsch graph on 16 vertices.

\begin{enumerate}[leftmargin=1cm, label=(\Roman*), resume=edu*]
\item For any fixed integer $k\ge 2$, let $G_k$ be the graph  defined as follows:

\vspace{-0.25cm}
\begin{itemize}[label=$\circ$]\itemsep=0pt
\item $V(G_k)=Q_1\cup Q_2\cup S$, where  $Q_1:=\{a_1,a_2,\ldots, a_k\}$, $Q_2:=\{b_1,b_2,\ldots,b_k\}$,
and $S:=\{s_1,s_2,$ $\ldots,s_k\}$ are cliques.

\item For each $i\in \{1,2,\ldots, k\}$, $a_i$ is adjacent to $b_i$, and $\{a_i\}$ is anticomplete to $Q_2\sm \{b_i\}$.
\item  For each $i\in \{1,2,\ldots, k\}$, $\{s_i\}$ is anticomplete to $\{a_i,b_i\}$, and  complete to $(Q_1\cup Q_2) \sm \{a_{i}, b_{i}\}$.
\item No other edges in $G$.
\end{itemize}
Then,  for any fixed integer $k\ge 2$,  $G_k$ is  $(P_2+P_3,\overline{P_2+P_3})$-free, $|V(G_k)|=3k$, $\omega(G_k)=k+1$, and  $\alpha(G_k)=2$; see also \cite{HK21}. Moreover, by Lemma~4 of \cite{HK21}, we have $\chi(G_k)\leq \lceil \frac{3}{2}k\rceil$.

\end{enumerate}

\smallskip
\no{\bf Proof of  \cref{thm:p2p3-bndopt}}. Clearly from   examples given above in (I) and (II), we may assume that $k\geq 10$; so $\max\{k+3,\lfloor \frac{3}{2}k\rfloor-1\} = \lfloor \frac{3}{2}k\rfloor-1$. Now, we show that $G_{k-1}$ is our desired graph $G$. Let $k-1=\ell$. Then $\omega(G_{\ell})=k$. Moreover, since $\chi(G_{\ell})\geq \frac{|V(G_{\ell})|}{\alpha(G_{\ell})}$ and  $\chi(G_{\ell})\leq   \lceil\frac{3}{2}\ell\rceil$, we conclude that $\chi(G_{\ell})=\lfloor \frac{3}{2}k\rfloor-1$. This proves \cref{thm:p2p3-bndopt}. \hfill{$\Box$}

%
%

 \medskip

The proof   of \cref{thm:p2p3-bnd} follows from our result for the class of ($P_2+ P_3$, $\overline{P_2+ P_3}$)-free graphs that contain a $C_4$, given below. To state it, we require some definitions.  A \emph{neighbor} (or \emph{nonneighbor}) of a vertex $v$ in $G$ is a vertex that is adjacent (or nonadjacent) to $v$ in $G$. The set of neighbors of a vertex $v$ in $G$ is denoted by $N_G(v)$, and $|N_G(v)|:=d_G(v)$ is the \emph{degree} of $v$.  We write $\overline{N}_G(v)$ to denote the set $V(G)\sm (N(v)\cup \{v\})$.
 Two  vertices $u$ and $v$ in a graph $G$ are said to be \emph{comparable}, if $u$ and $v$ are nonadjacent, and either $N_G(u)\subseteq N_G(v)$ or $N_G(v)\subseteq N_G(u)$. (We drop the
subscript $G$ when there is no ambiguity.) A vertex in a graph $G$ is a \emph{universal vertex}  if it is adjacent to all other vertices in $G$.
A vertex  $v$ in a graph $G$ is  a \emph{nice vertex} if   $d_G(v)\leq \omega(G)+2$.  We say that a graph $G$ is a \emph{nice} if it has three pairwise disjoint stable sets, say $S_1,S_2$ and $S_3$, such that $\omega(G-(S_1\cup S_2\cup S_3))\leq \omega(G)-2$. We say that a graph $G$ is \emph{good}, if  one of the following holds: $(a)$~$G$ has a pair of comparable vertices. $(b)$~$G$ has a universal vertex.  $(c)$~$G$ has a nice vertex.   $(d)$~$G$ is a nice graph.  $(e)$~$\chi(G)\leq \omega(G)+3$.

\begin{thm}
\label{thm:structure}
If $G$ is a ($P_2+ P_3$, $\overline{P_2+ P_3}$)-free graph that contains a $C_4$, then  $G$ is a good graph.
\end{thm}

The rest of the paper is organized as follows:
 In \cref{genprop}, we prove some  general properties of ($P_2+ P_3$, $\overline{P_2+P_3}$)-free graphs that contain a $C_4$, and use it in the later sections.  \cref{sec:thmstruc} is devoted to the proof of \cref{thm:structure}, and finally in \cref{sec:col}, we give a proof of \cref{thm:p2p3-bnd}.

\section{Properties of ($P_2+ P_3$, $\overline{P_2+P_3}$)-free graphs that contain a $C_4$}
\label{genprop}

Let $G$ be a  ($P_2+ P_3$, $\overline{P_2+P_3}$)-free graph that contains a $C_4$, say with vertex-set $C:=\{v_1,v_2,v_3,v_4\}$ and edge-set $\{v_1v_2,v_2v_3,v_3v_4,v_4v_1\}$.
 For $i\in \{1,2,3,4\}$, let $A_i$ denote the set $\{v\in V(G)\setminus C \mid N(v)\cap C=\{v_i\}\}$, and $B_i$ denote the set $\{v\in V(G)\setminus C \mid N(v)\cap C=\{v_{i},v_{i+1}\}\}$. Also, for $j\in \{1,2\}$, let $X_j$ denote the set $\{v\in V(G)\mid N(v)\cap C=\{v_j,v_{j+2}\}\}$. Moreover, let $D$  and $T$ respectively denote the set of vertices in $G$ which are complete to $C$,  and   anticomplete to $C$.  Let $A:=\cup_{i=1}^4A_i$, $B:=\cup_{i=1}^4B_i$, and $X:=X_1\cup X_2$.  Throughout the paper, our   indices are taken arithmetic modulo $4$ (unless stated otherwise). Since $G$ has no $\overline{P_2+P_3}$, no vertex in $V(G)\sm C$ is adjacent to three vertices in $C$, and hence $V(G)=A\cup B\cup C\cup D\cup X\cup T$. Further, we observe that the following hold:
\newcounter{ici}
\begin{enumerate}[leftmargin=1.25cm, label=(R\arabic*),series=edu*]

\item\label{ATstset}  {\it For  $i\in \{1,2,3,4\}$, $A_i\cup T$ is a stable set.}

\begin{proof2}
If there are adjacent vertices in $A_i\cup T$, say $p$ and $q$, then $\{p,q,v_{i+1},v_{i+2},v_{i+3}\}$ induces a $P_2+P_3$. \end{proof2}

\item\label{a1b1nonnbd} {\it For $i\in \{1,2,3,4\}$, any vertex in $A_i\cup B_i$ can have at most one nonneighbor in $A_{i+2}\cup B_{i+1}$. Likewise, any vertex in $A_{i+1}\cup B_i$ can have at most one nonneighbor in $B_{i-1}\cup A_{i-1}$.}
	
\begin{proof2}	Let  $p\in A_i\cup B_i$. If $p$ has two nonneighbors in $A_{i+2}\cup B_{i+1}$, say $q$ and $r$, then since $\{p,v_i,q,v_{i+2},r\}$ does not induce a $P_2+ P_3$, we may assume that $qr\in E(G)$, and then $\{q,r,p,v_i,v_{i+3}\}$ induces a $P_2+ P_3$.  \end{proof2}

	\item\label{abtnbBD} {\it For $i\in \{1,2,3,4\}$  and for any vertex $p\in A_i\cup B_i\cup T$, $N(p)\cap (D\cup B_{i+2})$ is a clique. Likewise, for any $p\in A_i$,  $N(p)\cap B_{i+1}$ is a clique.  Moreover, for any $p\in A_i$, $|N(p)\cap B_{i+2}|\leq 1$. Likewise, $|N(p)\cap B_{i+1}|\leq 1$.}
	
\begin{proof2}
If there are nonadjacent vertices, say $d_1, d_2\in N(p)\cap (D\cup B_{i+2})$, then $\{p, d_1,v_{i+2},d_2,v_{i+3}\}$ induces a $\overline{P_2+P_3}$, a contradiction. This proves the first assertion. Next, if there are vertices, say $b, b'\in N(p)\cap B_{i+2}$, then, by the first assertion, $\{p, v_i,v_{i+3}, b,b'\}$ induces a $\overline{P_2+P_3}$.   \end{proof2}

\item\label{bibi+2cmnnbd} {\it For $j\in \{1,2\}$, if there are adjacent vertices, say $b\in B_j$ and $b'\in B_{j+2}$, then $N(b)\cap (B_{j+1}\cup B_{j-1}\cup D)= N(b')\cap (B_{j+1}\cup B_{j-1}\cup D)$.}
		
\begin{proof2}		If there is a vertex, say $v\in N(b)\cap (B_{j+1}\cup B_{j-1}\cup D)$ such that $v\notin N(b')\cap (B_{j+1}\cup B_{j-1}\cup D)$, then, up to symmetry, we may assume that $v\in B_{j+1}\cup D$, and then $\{b,v_{j+1},v_{j+2},b',v\}$ induces a $\overline{P_2+P_3}$.  \end{proof2}
	
\item\label{Xstset} {\it For  $j\in\{1,2\}$, $X_j$ is a stable set.}	
	
\begin{proof2}If there are adjacent vertices in $X_j$, say $p$ and $q$,  then  $\{v_j,v_{j+1},v_{j+2},p,q\}$ induces a $\overline{P_2+P_3}$.  \end{proof2}
	
\item\label{BXantcomp}  {\it $B$ is anticomplete to $X$.}

\begin{proof2} By symmetry, it is enough to show that $B_1\cup B_2$ is anticomplete to $X_1$. If there are adjacent vertices, say $b\in B_1\cup B_2$ and $x\in X_1$, then  $\{v_1,v_2,v_3,x,b\}$ induces a $\overline{P_2+P_3}$.  \end{proof2}

 \item\label{Dper} {\it $G[D]$ is ($K_2+K_1$)-free, and hence perfect; and $\chi(G[D])= \omega(G[D])\leq \omega(G)-2$.}

 \begin{proof2} If  there are vertices, say $p,q,r\in D$ such that $\{p,q,r\}$ induces a $K_2+K_1$, then $\{p,q,r, v_2,v_4\}$ induces a $\overline{P_2+P_3}$; so $G[D]$ is ($K_2+K_1$)-free. Since $D$ is complete to $\{v_1,v_2\}$,   $\omega(G[D])\leq \omega(G)-2$, and hence $\chi(G[D]) =\omega(G[D])\leq \omega(G)-2$. \end{proof2}

	\item \label{XDcom} {\it $X$ is complete to $D$.}

\begin{proof2} If there are nonadjacent vertices, say $d\in D$ and $x\in X$, then, we may assume that $x\in X_1$, and then $\{v_1,v_{2},v_{3},x,d\}$ induces a $\overline{P_2+P_3}$. \end{proof2}

	\item\label{a1a3cmnnbda2} {\it For $j,k\in \{1,2\}$ and $j\neq k$, suppose there are adjacent vertices, say $p\in A_j$ and $q\in A_{j+2}$. Then:

\vspace{-0.2cm}
\begin{enumerate}[label=(\alph*)]\itemsep=0pt \item At most one vertex in $A_{j+1}$ is anticomplete to $\{p,q\}$. Likewise,  at most one vertex in $A_{j-1}$  is anticomplete to $\{p,q\}$.
\item  At most one vertex in $X_k$ is complete to $\{p,q\}$.
\item \label{X2DA1A3}  Each vertex of $D\cup X_k$ is adjacent to at least one of $p,q$. \end{enumerate}}

\begin{proof2}	We prove for $j=1$.
$(a)$~If there are vertices, say $r,s\in A_2$, such that $\{r,s\}$ is anticomplete to $\{p,q\}$, then, by \ref{ATstset}, $\{p,q,r,v_2,s\}$ induces a $P_2+P_3$. $(b)$~If there are vertices, say  $x,x'\in X_2$, such that $\{x,x'\}$ is complete to $\{p,q\}$, then, by \ref{Xstset}, $\{p,x,v_2,x',q\}$ induces a $\overline{P_2+P_3}$. $(c)$~If there is a vertex, say $r\in D\cup X_2$ such that  $pr,qr\notin E(G)$, then $\{p,q,v_2,r,v_4\}$ induces a $P_2+P_3$. This proves \ref{a1a3cmnnbda2}.
\end{proof2}
	
	\item\label{a1a2x} {\it For $i\in \{1,2,3,4\}$ and $j\in \{1,2\}$, if there are adjacent vertices, say $p\in A_i$ and $q\in A_{i+1}$, then each vertex of $X_j$ is adjacent to exactly one of $p$ and $q$. }

\begin{proof2} We prove for   $j=1$.		For any $x\in X_1$,  if $px,qx\in E(G)$, then $\{q,x,v_i,v_{i+1},p\}$ induces a $\overline{P_2+P_3}$, and if $px,qx\notin E(G)$, then $\{p,q,x,v_{i+2},v_{i+3}\}$ induces a $P_2+P_3$. \end{proof2}

	\item\label{AXonenbd} {\it For  $j\in\{1,2\}$,   any vertex in $A_j$ has at most one neighbor in $X_j$, and any vertex in $A_{j+2}$ has at most one neighbor in $X_j$.}

\begin{proof2}		If there is a vertex, say $a\in A_j$ which has two neighbors in $X_j$, say $x$ and $x'$, then, by \ref{Xstset}, $\{v_j,x,v_{j+2},x',a\}$ induces a $\overline{P_2+P_3}$. \end{proof2}

\item\label{prop-K23} {\it  For $i\in \{1,2,3,4\}$, the following hold:

\vspace{-0.2cm}
\begin{enumerate}[label=(\alph*)]\itemsep=0pt
\item\label{a1b1b3nbd} If $G$ is $K_{2,3}$-free,  then any vertex in $B_i$ has at most one neighbor in $A_{i-1}\cup B_{i+2}$. Likewise, any vertex in $B_i$ has at most one neighbor in $A_{i+2}\cup B_{i+2}$.
\item\label{A1A3ancom} If $G$ is ($K_2+K_3$)-free, then any vertex in $A_i\cup B_i$ has at most one nonneighbor in $B_{i+1}\cup B_{i+2}$. Likewise, any vertex in $A_{i+1}\cup B_i$ has at most one nonneighbor in $B_{i-1}\cup B_{i+2}$.
\end{enumerate}	}
\begin{proof2}
	$(a)$:~If there is a vertex, say $p\in B_i$ which has two neighbors in $A_{i-1}\cup B_{i+2}$, say $q$ and $r$, then $\{p,v_{i},v_{i-1},q,r\}$ induces   a $K_{2,3}$ or a $\overline{P_2+P_3}$.
$(b)$:~If there is a vertex, say $p\in A_i\cup B_i$ which has two nonneighbors in $B_{i+1}\cup B_{i+2}$, say $q$ and $r$, then $\{p,v_i,q,v_{i+2},r\}$ induces    a $K_2+K_3$ or a $P_2+P_3$. \end{proof2}

\item\label{A1A2onenbd} {\it For $i\in \{1,2,3,4\}$, if $G$ is ($K_2+K_3$)-free, then each vertex in $A_i\cup B_i\cup T$ can have at most one neighbor in $A_{i+1}\cup B_i$. Likewise, each vertex in $A_{i+1}\cup B_i$ can have at most one neighbor in $A_{i}\cup B_i\cup T$.}
	
	\begin{proof2}If   there is a vertex, say $p\in A_i\cup B_i\cup T$ which has two neighbors in $A_{i+1}\cup B_i$, say $q$ and $r$, then $\{v_{i+2},v_{i-1},p,q,r\}$ induces  a $P_2+P_3$ or a $K_2+K_3$.  \end{proof2}

	\item\label{ABst} {\it For $i\in \{1,2,3,4\}$, if $G$ is co-banner-free, then  $A_i\cup B_i$ and $A_{i+1}\cup B_i$ are stable sets.}
	
\begin{proof2}
	If there are adjacent vertices in $A_i\cup B_i$, say $p$ and $q$ , then $\{p,q,v_i,v_{i+3},v_{i+2}\}$ induces a co-banner.
\end{proof2}
	
	\item\label{AiB+2com} {\it For $i\in \{1,2,3,4\}$, if $G$ is co-banner-free, then $A_i\cup B_i\cup A_{i+1}$ is complete to $B_{i+2}$; and so if $A_i\cup B_i\cup A_{i+1}\neq \es$, then $|B_{i+2}|\leq 1$.}
	
\begin{proof2}
 If there are nonadjacent vertices, say $p\in A_i\cup B_i\cup A_{i+1}$ and $q\in B_{i+2}$, then $\{q,v_{i+2},v_{i+3}, $ $v_i, p\}$ or $\{q,v_{i+2},v_{i+3},v_{i+1}, p\}$ induces a co-banner; so $A_i\cup B_i\cup A_{i+1}$ is complete to $B_{i+2}$. Now, by \ref{abtnbBD}, $B_{i+2}$ is a clique, and  hence by \ref{ABst}, $|B_{i+2}|\leq 1$. \end{proof2}
\end{enumerate}

\section{($P_2+ P_3$, $\overline{P_2+ P_3}$)-free graphs that contain a $C_4$}\label{sec:thmstruc}
In this section, we give a proof of  \cref{thm:structure}. It  is based on a sequence of partial results which depend on some special graphs; see Figure~\ref{fig:partfig1}.  More precisely, given a   ($P_2+ P_3$, $\overline{P_2+ P_3}$)-free graph $G$, we will show that the following hold:

\vspace{-0.25cm}
\begin{enumerate}[label=(\roman*)]\itemsep=0pt
 \item If $G$ contains a $K_{2,3}$, then  $\chi(G)\leq \omega(G)+3$ (\cref{thm:k23});
  \item If $G$ is $K_{2,3}$-free, and contains a banner, then $G$ is a good graph (\cref{thm:banner});
\item If $G$ is ($K_{2,3}$, banner)-free, and contains an $H_2$, then   $G$ is a good graph (\cref{thm:5apple});
\item If $G$ is ($K_{2,3}$, banner, $H_2$)-free, and contains an $H_3$, then   $G$ is a good graph (\cref{thm:caseH3});
    \item  If $G$ is ($K_{2,3}$)-free, and contains a $C_4$, then   $G$ is a good graph (\cref{thm:c4}).
\end{enumerate}
 We remark that, to prove some of the above items, we will often consider the complement graph $\overline{G}$, and show that
$G$ is a good graph.

\begin{figure}[t]
\centering
 \includegraphics[height=2.5cm]{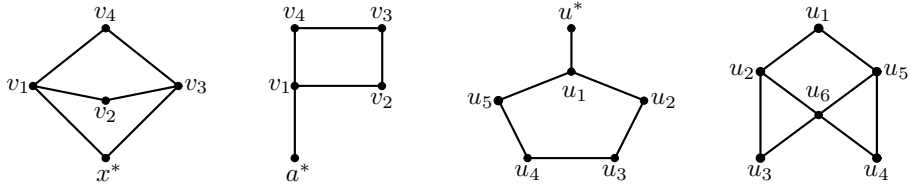}
\caption{A $K_{2,3}$, a banner, an $H_2$, and an $H_3$ (left to right).}\label{fig:partfig1}
\end{figure}

\subsection{($P_2+P_3$, $\overline{P_2+P_3}$)-free graphs that contain  a $K_{2,3}$}
\begin{thm}\label{thm:k23}
If $G$ is a ($\res$, $\overline{P_2+P_3}$)-free graph that contains a $K_{2,3}$, then $\chi(G)\leq \omega(G)+3$.
\end{thm}
\begin{proof}
Let $G$ be a ($\res$, $\overline{P_2+P_3}$)-free graph that contains a $K_{2,3}$. We may consider a $K_{2,3}$ with vertices and edges as in Figure~\ref{fig:partfig1}. Let $C:=\{v_1,v_2,v_3,v_4\}$.  We partition $V(G)\sm C$ as in Section~\ref{genprop},  and we use the properties in Section~\ref{genprop}. Note that by the definition of $X_1$, $x^*\in X_1$; so $X_1\neq \emptyset$.
Suppose that $\omega(G)=2$. Then, since $G$ is triangle-free, we have $B\cup D=\es$,  $A_1$ is anticomplete to $X_1$, and $A_2$ is anticomplete to $X_2$. Now we let $S_1:=A_1\cup X_1\cup \{v_2,v_4\}$, $S_2:=A_2\cup X_2\cup \{v_1,v_3\}$, $S_3:=A_3\cup T$, and $S_4:=A_4$. Then $V(G)=\cup_{i=1}^4S_i$. Clearly,  by \ref{ATstset} and \ref{Xstset}, $S_1,S_2,S_3$ and $S_4$ are stable sets. Thus  $\chi(G)\leq 4\leq \omega(G)+3$. So we may assume that $\omega(G)\geq 3$. Next, we have the following:

\begin{claim}\label{Bstset} For   $i\in \{1,2,3,4\}$, $B_i$ is a stable set.\end{claim}

\vspace{-0.2cm}
\no{\it Proof of \cref{Bstset}}. If there are adjacent vertices in $B_i$, say $p$ and $q$, then,  by \ref{BXantcomp}, $\{p,q,x^*, v_{i+2},v_{i+3}\}$ induces a $P_2+P_3$.   $\Diamond$

\begin{claim} \label{AiBiBi-1nancom} $A_2$ is anticomplete to $B_1\cup B_2$, and $A_4$ is anticomplete to $B_3\cup B_4$.  \end{claim}

\vspace{-0.2cm}
\no{\it Proof of \cref{AiBiBi-1nancom}}.~We prove that $A_2$ is anticomplete to $B_1$ in the first assertion.
Suppose there are adjacent vertices, say $a\in A_2$ and $b\in B_1$. By \ref{BXantcomp}, $bx^*\notin E(G)$. Now since $\{a,b,x^*,v_3,v_4\}$ does not induce a $\res$, we have $ax^*\in E(G)$, and then $\{a,x^*,v_1,v_2,b\}$ induces a $\overline{P_2+P_3}$.  $\Diamond$

\medskip
  First suppose that $\omega(G[D])\leq\omega(G)-3$. Now we let $S_1:=A_1\cup T\cup\{v_2,v_4\}$, $S_2:=B_1\cup X_1$, $S_3:=A_2\cup B_2$, $S_4:=A_3\cup\{v_1\}$, $S_5:=B_3\cup X_2$, and $S_6:=A_4\cup B_4\cup\{v_3\}$. Then $V(G)\setminus D=\cup_{i=1}^6S_i$. Also, by \ref{ATstset}, \ref{Xstset}, \ref{BXantcomp}, and by \cref{Bstset,AiBiBi-1nancom}, $S_1,S_2,\ldots, S_6$ are stable sets.    Hence $\chi(G)\leq \chi(G[D])+6\leq (\omega(G)-3)+6 = \omega(G)+3$, and we are done.
So, by \ref{Dper}, we may assume  that $\omega(G[D])=\omega(G)-2$. Then since $\omega(G)\geq 3$, $D\neq \emptyset$.
Let $A_1'$ denote the set $\{a\in A_1\mid  a \mbox{ has a neighbor in } X_1\}$. Then we have the following.
\begin{claim}\label{B1acB2} $B_1$ is anticomplete to $B_2$. Likewise $B_3$ is anticomplete to $B_4$.
\end{claim}

\vspace{-0.2cm}
\no{\it Proof of \cref{B1acB2}}.~Suppose to the contrary that there are adjacent vertices, say $b_1\in B_1$ and $b_2\in B_2$. Note that for any vertex $d\in D$, by \ref{XDcom}, since $\{b_1,b_2,x^*,d,v_4\}$ does not induce a $\res$, $d$ is adjacent to one of $b_1$, $b_2$. We let $D_1:=\{d\in D\mid db_1\in E(G),db_2\notin E(G)\}$, $D_2:=\{d\in D\mid db_2\in E(G),db_1\notin E(G)\}$ and $D_3:=\{d\in D\mid db_1,db_2\in E(G)\}$. Then $D=D_1\cup D_2\cup D_3$. Now if there are adjacent vertices, say $d_1\in D_1$ and $d_2\in D_2$, then $\{b_1,d_1,d_2,b_2,v_3\}$ induces a $\overline{P_2+P_3}$; so $D_1$ is anticomplete to $D_2$. Moreover, by \ref{abtnbBD}, it follows that  $ D_1\cup D_3 (=N(b_1)\cap D)$  and  $D_2\cup D_3 (=N(b_2)\cap D)$ are cliques. Thus we conclude that any maximum clique in $G[D]$ is  either $D_1\cup D_3$ or $D_2\cup D_3$; so $\max\{|D_1\cup D_3|, |D_2\cup D_3|\}=\omega(G)-2$.  Then since $D_1\cup D_3\cup\{b_1,v_1,v_{2}\}$ and $D_2\cup D_3\cup\{b_2,v_2,v_{3}\}$ are cliques,  and so $\max\{|D_1\cup D_3\cup\{b_1,v_1,v_{2}\}|, |D_2\cup D_3\cup\{b_2,v_2,v_{3}\}|\} =(\omega(G)-2)+3=\omega(G)+1$, a contradiction. This proves \cref{B1acB2}. $\Diamond$

 \begin{claim}\label{A1'-ac-X2}  $A_1'$ is anticomplete to $X_2$.\end{claim}

\vspace{-0.2cm}
\no{\it Proof of \cref{A1'-ac-X2}}.~Suppose to the contrary that there are adjacent vertices, say $a\in A_1'$ and $x\in X_2$. By the definition of $A_1'$, there is a vertex $x'\in X_1$  such that $ax'\in E(G)$. Recall that, by \ref{XDcom}, $D$ is complete to $X$. Also, for any $d\in D$, since  $\{a,v_1,v_2,x,d\}$ does not induce a $\overline{P_2+P_3}$; so $D$ is complete to $\{a\}$. Thus, by \ref{abtnbBD}, $D$ is a clique.  Now  $D\cup\{a,v_1,x'\}$ is a clique of size $\omega(G)+1$, a contradiction. This proves \cref{A1'-ac-X2}. $\Diamond$

\smallskip
Now we let $S_1:=A_2\cup B_1\cup B_2\cup \{v_4\}$, $S_2:=A_4\cup B_3\cup B_4\cup \{v_2\}$, $S_3:=A_1'\cup X_2\cup\{v_3\}$, $S_4:=(A_1\sm A_1')\cup X_1$, and $S_5:=A_3\cup T\cup \{v_1\}$. Then $V(G)\setminus D=\cup_{i=1}^5S_i$. Also, by \cref{Bstset,AiBiBi-1nancom,B1acB2,A1'-ac-X2}, by \ref{ATstset}, and  by \ref{Xstset}, we see that $S_1,S_2,\ldots, S_5$ are stable sets. So $\chi(G)\leq \chi(G[D])+5$.  Then, by \ref{Dper},   we have $\chi(G)\leq \omega(G)+3$.  This completes the proof.
\end{proof}


%
%
%
%
%
%
%
%
%
%

 \begin{figure}[t]
\centering
 \includegraphics[height=2.5cm]{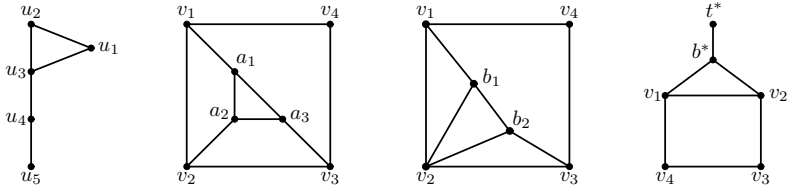}
\caption{A co-banner, an $H_1$, an $\overline{H_2}$, and an $\overline{H_3}$ (left to right).}\label{fig:partfig2}
\end{figure}

\subsection{($P_2+P_3$, $\overline{P_2+P_3}$, $K_{2,3}$)-free graphs that contain a banner}

We consider two cases  depending on whether ($P_2+P_3$, $\overline{P_2+P_3}$, $K_{2,3}$)-free graphs contain an $\overline{H_1}$ or not, and are given below in two subsections.

\subsubsection{($P_2+P_3$, $\overline{P_2+P_3}$, $K_2+K_3$)-free graphs that contain an $H_1$}
Let $G$ be a $(P_2+P_3$, $\overline{P_2+P_3}$, $K_2+K_3)$-free graph that contains an $H_1$. We may consider an $H_1$  with vertices and edges as in Figure~\ref{fig:partfig2}. Let $C:=\{v_1,v_2,v_3,v_4\}$.  We partition $V(G)\sm C$ as in Section~\ref{genprop},  and we use the properties in Section~\ref{genprop}. Clearly, $a_1\in A_1$, $a_2\in A_2$ and $a_3\in A_3$. To proceed further, let $L_1:=\{a\in A_4\mid aa_2\in E(G)  \mbox{ and } N(a)\cap \{a_1,a_3\}\neq\es\}$, $L_2:=\{a\in A_4\mid aa_2\in E(G)  \mbox{ and } N(a)\cap \{a_1,a_3\}=\es\}$,  and let $L_3:=\{a\in A_4\mid aa_2\notin E(G)\}$. Then clearly $A_4=\cup_{j=1}^3L_j$.  By   \ref{a1a3cmnnbda2}:(a), $|L_2|\leq 1$, and, by \ref{a1b1nonnbd}, $|L_3|\leq 1$; so if $L_1=\es$, then $|A_4|\leq 2$.  Moreover, the graph $G$ has some more properties which we give in \cref{G1-A123,G1-BD-theta,H1-theta,H1-L1} below.
\begin{lemma}\label{G1-A123} The following hold:

\vspace{-0.4cm}
\begin{lemmalist}
 \item\label{G1-A123i} $A_1\sm \{a_1\}$ is complete to $\{a_3\}$. Likewise, $A_3\sm \{a_3\}$ is complete to $\{a_1\}$.
\item\label{G1-A123ii}
 $(A_1\sm \{a_1\})\cup (A_3\sm \{a_3\})$ is anticomplete to $\{a_2\}$, and is complete to $(A_2\sm\{a_2\})\cup L_2$.
  \item\label{G1-A123iii} If $(A_2\sm \{a_2\})\cup L_2\neq \es$, then  $|A_1\sm \{a_1\}|\leq 1$ and $|A_3\sm \{a_3\}|\leq 1$.
\item\label{H1-X2emp} $X_2$ is empty.
\end{lemmalist}
\end{lemma}
\begin{proof} $(i)$:~For any $a_1'\in A_1\sm \{a_1\}$,  by \ref{A1A2onenbd}, $a_2a_1'\notin E(G)$, and then since $\{a_2,a_3, a_1',v_1,v_4\}$ does not induce a $P_2+P_3$, we have $a_1'a_3\in E(G)$; so $A_1\sm \{a_1\}$ is complete to $\{a_3\}$. $\diamond$

\no{$(ii)$:}~Clearly, the first assertion of $(ii)$ follows from  \ref{A1A2onenbd}.  If there are nonadjacent vertices, say $a_1'\in A_1\sm \{a_1\}$     and $a_2'\in A_2\sm \{a_2\}$, then, by \ref{A1A2onenbd} and \cref{G1-A123i},   $\{a_2',v_2,a_1,a_3,a_1'\}$  induces a $P_2+P_3$; so $A_1\sm \{a_1\}$ is complete to $A_2\sm\{a_2\}$. Likewise, $A_3\sm \{a_3\}$ is complete to $A_2\sm\{a_2\}$. If there are nonadjacent vertices, say $a'\in A_1\sm \{a_1\}$ and $a\in L_2$, then, by \cref{G1-A123i},  $\{a,v_4,a_1,a_3,a_1'\}$ induces a $P_2+P_3$; so $A_1\sm \{a_1\}$ is complete to $L_2$. Likewise, $A_3\sm \{a_3\}$ is complete to $L_2$. $\diamond$

\no{$(iii)$:}~This  follows from \cref{G1-A123ii} and \ref{A1A2onenbd}. $\diamond$

\no{$(iv)$:} If there is a vertex, say $x\in X_2$, then, by \ref{a1a3cmnnbda2}:\ref{X2DA1A3}, we may assume that $a_1x\in E(G)$, and then, by \ref{a1a2x}, $a_2x\notin E(G)$ and $a_3x\in E(G)$, and hence $\{a_1,a_2,v_2,x,a_3\}$ induces a $\overline{P_2+P_3}$; so $X_2 =\es$.
\end{proof}

\begin{lemma}\label{G1-BD-theta}The following hold:

\vspace{-0.4cm}
\begin{lemmalist}
\item   \label{B1B2card}$B_1$ is  anticomplete to $\{a_1,a_2\}$, and $B_2$ is anticomplete to $\{a_2,a_3\}$. Moreover,
$B_1$ is complete to $\{a_3\}$, and $B_2$ is complete to $\{a_1\}$; so, $|B_1| \leq 1$ and $|B_2|\leq 1$.
\item\label{H1-B3B4card}  $|B_3|\leq 2$, and  $|B_4|\leq 2$. Further, if $|B_3| = 2$, then $|B_4|\leq 1$ and vice versa.
\item\label{D-a123}   $D$ is complete to $\{a_1,a_2,a_3\}$, and $D$ is a clique.
\end{lemmalist}
\end{lemma}
\begin{proof} $(i)$:~The first assertion follows from \ref{A1A2onenbd}.
Now for any $b\in B_1$, since $\{a_2,a_3,b,v_1,$ $v_4\}$ does not  induce a $P_2+P_3$, $B_1$ is complete to $\{a_3\}$; so, by \ref{abtnbBD}, $|B_1|\leq 1$. Likewise, $B_2$ is complete to $\{a_1\}$, and $|B_2|\leq 1$. $\diamond$

\no{$(ii)$}:~By \ref{a1b1nonnbd}, $a_2$ can have at most one nonneighbor in $B_3\cup B_4$. By \ref{abtnbBD}, $a_2$ can have at most one neighbor in $B_3$. Likewise,  $a_2$ can have at most one neighbor in $B_4$. So $|B_3|\leq 2$ and  $|B_4|\leq 2$,  and if $|B_3| = 2$ then $|B_4|\leq 1$ and vice versa. $\diamond$

\no{$(iii)$:}~Let $d\in D$. Then since $\{v_1,a_1,a_2,v_2,d\}$ does not induce a $\overline{P_2+P_3}$,  we have $da_1,da_2\notin E(G)$ or $da_1,da_2\in E(G)$. If $da_1,da_2\notin E(G)$, then $\{a_1,a_2,d,v_3,v_4\}$ induces a $K_2+K_3$. So, we have $da_1,da_2\in E(G)$. Likewise, $da_3\in E(G)$. Hence $D$ is complete to $\{a_1,a_2,a_3\}$. Then, by \ref{abtnbBD},   $D$ is a clique.
\end{proof}

\begin{lemma}\label{H1-theta} The following hold:

\vspace{-0.4cm}
\begin{lemmalist}
\item\label{theta-BC}
 $\theta(G[B\cup C])\leq 3$.
 \item\label{a1a2bcdclq}	 $\theta(G[L_3\cup B\cup C\cup D \cup \{a_1,a_2,a_3\}])\leq 4$.
 \end{lemmalist}
 \end{lemma}

 \begin{proof}
 \no{$(i)$:}~Consider the graph $G[B\cup C]$. By \cref{H1-B3B4card}, either $|B_3|\leq 1$ or $|B_4|\leq 1$. We may assume that $|B_4|\leq 1$.
 First suppose that there are nonadjacent vertices, say $b_3\in B_3$ and $b_2\in B_2$. Then by \ref{prop-K23}:\ref{A1A3ancom}, $\{b_2\}$ is complete to $(B_3\sm \{b_3\})\cup B_4$. Then, by \ref{bibi+2cmnnbd}, $B_4$ is complete to $(B_3\sm \{b_3\})$.
 So, $\{b_2\}\cup (B_3\sm \{b_3\})\cup B_4$, $ \{b_3, v_3,v_4\}$, $B_1\cup \{v_1,v_2\}$ are cliques, and thus $\theta(G[B\cup C])\leq 3$. Hence we may assume that $B_3$ is complete to $B_2\cup B_4$. Then, clearly  $\theta(G[B_2\cup B_3\cup B_4\cup \{v_3,v_4\}])\leq 2$, and since $B_1\cup \{v_1,v_2\}$ is a clique, $\theta(G[B\cup C])\leq 3$.   $\diamond$

 \no{$(ii)$}:~If $L_3= \es$, then since $D\cup \{a_1,a_2,a_3\}$ is a clique, clearly, by \cref{theta-BC}, $\theta(G[L_3\cup B\cup C\cup D \cup \{a_1,a_2,a_3\}])\leq \theta(G[B\cup C])+1\leq 4.$  So we assume that $L_3\neq \es$, and let $L_3:=\{a^*\}$. Then since $\{a_1,a_2,a^*,v_4,v_3\}$ does not induce a  $P_2+P_3$, $a_1a^*\in E(G)$. Likewise, $a_3a^*\in E(G)$. Then, as in the proof of \cref{B1B2card}, we have  $|B_3|\leq 1$ and $|B_4|\leq 1$; $\{a^*\}$ is complete to $B_1$, and $\{a_1,a_2\}$ is complete to $B_3$.
Also, if there are nonadjacent vertices, say $d\in D$ and $b\in B_3$,  then, by \cref{D-a123}, $\{b, a_1, d,v_4,v_1\}$ induces a $\overline{P_2+P_3}$; so $D$ is complete to $B_3$. Thus, $\{a^*\}\cup B_1\cup\{a_3\}$, $B_3\cup D\cup \{a_1,a_2\}$, $B_2\cup \{v_2,v_3\}$ and $B_4\cup \{v_1,v_4\}$ are cliques, and hence $\theta(G[L_3\cup B\cup C\cup D \cup \{a_1,a_2,a_3\}])\leq 4$. \end{proof}

\begin{lemma}\label{H1-L1}
If $L_1\neq \es$, then $\theta(G)\leq \alpha(G)+3$.
\end{lemma}
\begin{proof}
Let $a_4\in L_1$. We may assume that  $a_1a_4\in E(G)$. Then as in the proof of \cref{H1-X2emp}, we have $X_1=\es$. Since $X_2=\es$ (by \cref{H1-X2emp}), by \cref{theta-BC}, it is enough to show that $\theta(G[A\cup D\cup T])\leq \alpha(G)$.
 Recall that $D$ is complete to $\{a_1,a_2,a_3\}$, and $D$ is a clique. Also, as in the proof of \cref{D-a123}, $D$ is complete to $\{a_4\}$.
 Now if $(A_2\sm \{a_2\})\cup (A_4\sm \{a_4\})=\es$, then by \ref{ATstset}, $A_1\cup A_3\cup T$ induces a bipartite graph, and is anticomplete to $\{v_2,v_4\}$; so  $\theta(G[A_1\cup A_3\cup T])\leq \alpha(G)-2$, and then, since $D\cup \{a_2,a_4\}$ is a clique, we conclude that $\theta(G[A\cup D\cup T])\leq \alpha(G)$. So we assume that $(A_2\sm \{a_2\})\cup (A_4\sm \{a_4\})\neq \es$. Note that, by \ref{A1A2onenbd}, $A_1\sm \{a_1\}$ is anticomplete to $\{a_2,a_4\}$; so by \ref{a1a3cmnnbda2}:(a), $|A_1\sm \{a_1\}|\leq 1$.  Also, if $|A_3|\geq 3$, then by \cref{G1-A123iii}, it follows that $A_4\sm\{a_4\}\neq \es$, and so for any $a\in A_4\sm\{a_4\}$,   by \ref{A1A2onenbd},  there are vertices $p,q\in A_3$ such that $ap,aq\notin E(G)$, and hence, again by \ref{A1A2onenbd} and \cref{G1-A123i}, $\{a,v_4,p,a_1,q\}$ induces a $P_2+P_3$; so $|A_3|\leq 2$. Next, if there are nonadjacent vertices, say $a_1'\in A_1\sm \{a_1\}$ and $d\in D$, then, by \cref{G1-A123i}, $\{v_1,a_1,a_3,a_1',d\}$ induces a $\overline{P_2+P_3}$; so  $D$ is complete to $A_1\sm \{a_1\}$. Likewise,  $D$ is complete to $A_3\sm \{a_3\}$.  Hence, by \cref{G1-A123i} and \cref{D-a123}, $\theta(G[A_1\cup A_3\cup D])\leq 2$. Also, by \ref{ATstset}, $A_2\cup A_4\cup T$ induces a bipartite graph, and is anticomplete to $\{v_1,v_3\}$; so  $\theta(G[A_2\cup A_4\cup T])\leq \alpha(G)-2$. Thus, $\theta(G[A\cup D\cup T])\leq \alpha(G)$. This proves \cref{H1-L1}.
\end{proof}

\begin{thm} \label{thm:H1c}
If $G$ is a $(P_2+P_3$, $\overline{P_2+P_3}$, $K_2+K_3)$-free graph that contains an  $H_1$, then $\theta(G)\leq \alpha(G)+3$.
\end{thm}
\begin{proof}
Let $G$ be a $(P_2+P_3$, $\overline{P_2+P_3}$, $K_2+K_3)$-free graph than contains an    $H_1$. We may consider an $H_1$ with vertices and edges as in Figure~\ref{fig:partfig2}. Let $C:=\{v_1,v_2,v_3,v_4\}$.  We partition $V(G)\sm C$ as in Section~\ref{genprop},  and we use the properties in Section~\ref{genprop}.  We partition $A_4$ as above, and we use \cref{G1-A123,G1-BD-theta,H1-theta,H1-L1}. Recall that, by \cref{H1-X2emp}, $X_2=\es$. By   \ref{a1a3cmnnbda2} and \ref{A1A2onenbd}, we have $|A_2|\leq 2$.
 By \cref{H1-L1},  we may assume that $L_1=\es$; so $|A_4|\leq 2$.  Now we prove the theorem in two cases as follows:

\smallskip
\no{\bf Case~1}~{\it $L_2\cup (A_2\sm \{a_2\})\neq \es$.}

First suppose that $(A_1\sm \{a_1\})\cup (A_3\sm \{a_3\})= \es$. Define $X_1':=\{x\in X_1\mid a_1x\in E(G)\}$. Then it follows from \ref{a1a2x} that $X_1'$ is complete to $\{a_3\}$; so by \ref{AXonenbd}, $|X_1'|\leq 1$. Note that $T\cup (X\sm X_1')$ induces a bipartite graph (by \ref{ATstset} and \ref{Xstset}), and is anticomplete to $\{a_1,v_2,v_4\}$; so $\theta(G[T\cup (X\sm X_1')])\leq \alpha(G)-3$. Since $|A_2|\leq 2$ and $|A_4|\leq 2$, by \ref{a1b1nonnbd}, $\theta(G[A_2\cup A_4])\leq 2$. By \ref{XDcom} and \cref{D-a123}, $\theta(G[A_1\cup A_3\cup D\cup X_1'])\leq 1$. Hence, by \cref{theta-BC}, $\theta(G)\leq \alpha(G)+3$.

Next suppose that $(A_1\sm \{a_1\})\cup (A_3\sm \{a_3\})\neq \es$. By \cref{G1-A123iii}, $|A_1\sm \{a_1\}|\leq 1$ and $|A_3\sm \{a_3\}|\leq 1$. Thus, by \cref{G1-A123ii}, $(A_2\sm \{a_2\})\cup (A_3\sm \{a_3\})$  and $(A_1\sm \{a_1\})\cup L_2$ are cliques. So, by \cref{a1a2bcdclq}, $\theta(G[A\cup B\cup C\cup D])\leq 6$. Next:

 \begin{claim}\label{H1-X1a} $X_1$ is anticomplete to $\{a_1\}$. \end{claim}

\vspace{-0.2cm}
\no{\it Proof of \cref{H1-X1a}}.~Suppose that there is a vertex $x\in X_1$ such that $xa_1\in E(G)$.  By \ref{a1a2x}, $a_2x\notin E(G)$ and $a_3x\in E(G)$.
Let $a'\in (A_1\sm \{a_1\})\cup (A_3\sm \{a_3\})$. Then, by \cref{G1-A123i},  since $\{a_1,v_1,a',a_3, x\}$  or $\{a_1,a_3,v_3,a',x\}$ does not induce a $\overline{P_2+P_3}$, we have $a'x\in E(G)$.  Then,  by \cref{G1-A123ii} and by \ref{a1a2x},  $\{x\}$ is anticomplete to $(A_2\sm \{a_2\})\cup L_2$. But, then for any $a^*\in (A_2\sm \{a_2\})\cup L_2$, one of $\{a^*,v_2, a_1,a_3,x\}$, $\{a^*,v_4, a_1,a_3,x\}$ induces a $K_2+K_3$. So  $X_1$ is anticomplete to $\{a_1\}$. $\Diamond$

   By \cref{H1-X1a}, since $T\cup X_1$ induces a bipartite graph (by \ref{ATstset} and \ref{Xstset}), and is anticomplete to $\{a_1,v_2,v_4\}$,  $\theta(G[T\cup X_1])\leq \alpha(G)-3$. Hence $\theta(G)\leq \theta(G[A\cup B\cup C\cup D])+\theta(G[T\cup X_1])\leq 6+(\alpha(G)-3)=\alpha(G)+3$.

 \smallskip
\no{\bf Case~2}~{\it $L_2\cup (A_2\sm \{a_2\})=\es$.}

If $(A_1\sm \{a_1\})\cup (A_3\sm \{a_3\})$ is a stable set, then since  $A_1\sm \{a_1\}\cup (A_3\sm \{a_3\})\cup T\cup X_1$ induces a bipartite graph (by \ref{ATstset} and \ref{Xstset}), and is anticomplete to $\{v_2,v_4\}$, we see that $\theta(G)\leq \theta(G[A_1\sm \{a_1\}\cup (A_3\sm \{a_3\})\cup X_1\cup T]) \leq \alpha(G)-2$, and we conclude the proof by using \cref{a1a2bcdclq}. So  we may assume that,  by \ref{ATstset}, there are adjacent vertices, say $a_1'\in A_1\sm \{a_1\}$ and $a_3'\in A_3\sm \{a_3\}$.

\vspace{-0.25cm}
\begin{claim}\label{X1le2} $X_1$ is complete to $ \{a_2\}\cup L_3$. Moreover, $|X_1|\leq 2$.
\end{claim}

\vspace{-0.2cm}
\no{\it Proof of \cref{X1le2}}.~  Suppose there is a vertex $x\in X_1$ such that $xa_2\notin E(G)$.  Then, by \ref{a1a2x}, $a_1x,a_3x\in E(G)$. Then, by \cref{G1-A123i}, since $\{v_1,a_1,a_3,a_1',x\}$ does not induce a $\overline{P_2+P_3}$, $a_1'x\in E(G)$. But then, by \cref{G1-A123ii},  $\{a_2,v_2,a_1',x,a_3'\}$ induces   a  $K_2+K_3$ or  a $P_2+P_3$. So we  conclude that $X_1$ is complete to $\{a_2\}$. Hence, if there are nonadjacent vertices, say $x\in X_1$ and $a\in L_3$, then $\{a,v_4,x,a_2,v_2\}$ induces a $P_2+P_3$; so $X_1$ is complete to $L_3$. Next, if $|X_1|\geq 3$, by \ref{AXonenbd}, there is a vertex $x\in X_1$ such that $xa_1',xa_3'\notin E(G)$, then, by \cref{G1-A123ii}, $\{a_1',a_3',x,a_2,v_2\}$ induces a $P_2+P_3$; so  $|X_1|\leq 2$. $\Diamond$

Now, by \cref{X1le2}, by \cref{D-a123} and \ref{XDcom}, $\theta(G[A_2\cup L_3\cup D\cup X_1])\leq 2$. Since $A_1\cup A_3\cup T$ induces a bipartite graph (by \ref{ATstset}), and is anticomplete to $\{v_2,v_4\}$, $\theta(G[A_1\cup A_3\cup T])\leq \alpha(G)-2$, so by \cref{theta-BC}, we are done. This completes the proof of \cref{thm:H1c}.
\end{proof}

\subsubsection{($P_2+P_3$, $\overline{P_2+P_3}$, $K_{2,3}$, $\overline{H_1}$)-free graphs that contain a banner}

\begin{thm}\label{thm:banner}
If $G$ is a ($P_2+P_3$, $\overline{P_2+P_3}$, $K_{2,3})$-free graph  that contains a  banner, then $G$ is a good graph.
\end{thm}
\begin{proof}Let $G$ be a ($P_2+P_3$, $\overline{P_2+P_3}$, $K_{2,3})$-free graph  that contains a banner. We may consider a banner  with vertices and edges as in Figure~\ref{fig:partfig1}. Let $C:=\{v_1,v_2,v_3,v_4\}$.  We partition $V(G)\sm C$ as in Section~\ref{genprop},  and we use the properties in Section~\ref{genprop}. Since $G$ is $K_{2,3}$-free, clearly $X=\es$. Recall that, by the definition of $A_1$, we have $a^*\in A_1$, and so $A\neq \es$. We may assume that $G$ has no pair of comparable vertices.  Moreover, by \cref{thm:H1c}, we may assume that $G$ is $\overline{H_1}$-free. Also:

\begin{claim}\label{AB1B2bipart}
 Suppose that $A_4\neq \es$. Then the set $A_1\cup A_4\cup B_1\cup \{v_3\}$ can be partitioned into two stable sets. Likewise, $A_3\cup A_4\cup B_2\cup \{v_1\}$ can be partitioned into two stable sets.
\end{claim}

\vspace{-0.2cm}
\no{\it Proof of \cref{AB1B2bipart}}.~ If $A_4$ is anticomplete to $B_1$, then by \ref{a1b1nonnbd}, $|A_4|=1$ and $|B_1|\leq 1$, and so $A_1\cup \{v_3\}$ and $A_4\cup B_1$ are stable sets, and we are done. So, we may assume that there are adjacent vertices, say  $a_4\in A_4$ and $b_1\in B_1$.
Then, by \ref{abtnbBD} and \ref{prop-K23}:\ref{a1b1b3nbd}, $\{a_4\}$ is anticomplete to $B_1\sm \{b_1\}$, and $A_4\sm \{a_4\}$ is anticomplete to $\{b_1\}$. So $(A_4\sm \{a_4\})\cup \{b_1\}$ is a stable set, and by \ref{a1b1nonnbd}, $|B_1\sm \{b_1\}|\le 1$. Now we show that $A_1\cup (B_1\sm \{b_1\})\cup \{a_4,v_3\}$ is a stable set. First, if there is a vertex $a_1\in A_1$ such that $a_1a_4\in E(G)$, then $\{a_1,a_4,v_4,v_1,b_1\}$ induces   a $\overline{P_2+P_3}$ or  a $K_{2,3}$; so $A_1\cup \{a_4\}$ is a stable set. Next, if there are adjacent vertices, say $a\in A_1$ and $b\in B_1\sm\{b_1\}$, then $\{a_4,v_4,a,b,v_2\}$ induces a $P_2+P_3$. Thus, $A_1\cup (B_1\sm \{b_1\})\cup \{a_4,v_3\}$ is a stable set. This proves Claim~\ref{AB1B2bipart}. $\Diamond$

Now we split the proof into two cases based on the subsets of $B$.

\smallskip
\no{\bf Case~1}  {\it 	Suppose that $B_i$ and $B_{i+2}$ are nonempty, for some $i\in \{1,2,3,4\}$.}

We let $i=1$, and we claim the following:

\vspace{-0.2cm}
\begin{claim}\label{a1b1g1forbid} $B_1$ is complete to $B_3$, and $B_2$ is complete to $B_4$.
\end{claim}

\vspace{-0.2cm}
 \no{\it Proof of \cref{a1b1g1forbid}}.~Suppose there are nonadjacent vertices, say $b\in B_1$ and $b'\in B_3$. Then since  $\{b',v_3,a^*,v_1,b\}$ does not induce a $P_2+P_3$,   either $ba^*\in E(G)$ or $b'a^*\in E(G)$.
 If $ba^*\in E(G)$, then since $\{b',v_4,a^*,b,v_2\}$ does not induce a $P_2+P_3$, we have $b'a^*\in E(G)$, and then  $\{a^*,b,b'\}\cup C$ induces an $\overline{H_1}$. So we may assume that $ba^*\notin E(G)$ and $b'a^*\in E(G)$. But then $\{b,v_2,a^*,b',v_4\}$ induces a $P_2+P_3$. So $B_1$ is complete to $B_3$. Likewise,  $B_2$ is complete to $B_4$. This proves Claim~\ref{a1b1g1forbid}. $\Diamond$

By \ref{prop-K23}:\ref{a1b1b3nbd} and Claim~\ref{a1b1g1forbid}, we have $|B_1|= 1= |B_3|$, $A_1$ is anticomplete to $B_3$, and $A_3$ is anticomplete to $B_1$. Thus $S_1:=A_1\cup B_3\cup \{v_2\}$ and $S_2:=A_3\cup B_1\cup \{v_4\}$ are stable sets.  Also, by \ref{ATstset}, $S_3:=A_2\cup T\cup \{v_1,v_3\}$ is a stable set. Now, if $B_2,B_4\neq \es$ or if  $B_2\cup B_4 =\es$, then as in the previous argument, $B_4$  and $A_4\cup B_2$ are stable sets, and hence    $\chi(G[V(G)\sm D])\leq 5$. Then, by \ref{Dper}, we conclude that $\chi(G)\leq \omega(G)+3$. So we may assume that   $B_2\neq \es$  and $B_4=\es$. If $A_4\neq \es$, by \cref{AB1B2bipart}, $A_4\cup B_2$ can be partitioned into two stable sets, and so  $\chi(G[V(G)\sm D])\leq 5$. Then again by \ref{Dper},  we conclude that $\chi(G)\leq \omega(G)+3$. So we may assume that $A_4=\es$. Now note that, since $A_4\cup B_4=\es$, $V(G)\sm (S_1\cup S_2\cup S_3)=B_2\cup D$. Then, since $\{v_2,v_3\}$ is complete to $B_2\cup D$, we have $\omega(G[B_2\cup D]) \leq \omega(G)-2$, and hence $G$ is nice. This proves the theorem in Case~1.

\medskip
\no{\bf Case~2}  {\it Suppose that, for each $i\in \{1,2,3,4\}$, at least one of $B_i$, $B_{i+2}$ is empty.}

We may assume that $B_3\cup B_4=\es$. Since $N(v_2)=A_2\cup B_1\cup B_2\cup D\cup \{v_1,v_3\}$ and $N(v_4)=A_4\cup D\cup \{v_1,v_3\}$, and since $G$ has no pair of comparable vertices, we have $A_4\neq \es$; let $a_4\in A_4$.
Now if $A_2=\es$, then since  $T\cup \{v_2,v_4\}$ is a stable set (by \ref{ATstset}), by \cref{AB1B2bipart}, $\chi(G[V(G)\sm D])\leq 5$, and so, by \ref{Dper}, we conclude that $\chi(G)\leq \omega(G)+3$. So we may assume that $A_2\neq \es$. First suppose that either $\{a_4\}$ is anticomplete to $A_2$ or $\{a_4\}$ is anticomplete to $B_1$.
Now if $\{a_4\}$ is anticomplete to $A_2$, then by \ref{a1b1nonnbd}, $\{a_4\}$ is complete to $B_1$; so by \ref{prop-K23}:\ref{a1b1b3nbd}, $|B_1|\leq 1$, and if $\{a_4\}$ is anticomplete to $B_1$, then by \ref{a1b1nonnbd}, again $|B_1|\leq 1$. In any case, by \ref{ATstset}, $A_2\cup B_1\cup \{v_3\}$ induces a bipartite graph. Since $A_1\cup T\cup \{v_2,v_4\}$ is a stable set (by \ref{ATstset}), by \cref{AB1B2bipart}, $\chi(G[V(G)\sm D])\leq 5$, and so by \ref{Dper}, we have $\chi(G)\leq \omega(G)+3$.
So we may assume that there are vertices $a_2\in A_2$ and $b_1\in B_1$ such that $a_2a_4,a_4b_1\in E(G)$. Then, by \ref{a1b1nonnbd}, $B_1\sm \{b_1\}$ is a stable set. Now, for any $d\in D$, since $\{a_4,b_1,v_1,v_4,d\}$ does not induce a $\overline{P_2+P_3}$, $\{b_1\}$ is anticomplete to $D\sm N(a_4)$, and $\{b_1\}$ is complete to $N(a_4)\cap D$.   By \ref{a1a3cmnnbda2}:\ref{X2DA1A3}, $D\sm N(a_4)$ is complete to $\{a_2\}$, and hence by \ref{abtnbBD}, $D\sm N(a_4)$ and $N(a_4)\cap D$ are cliques. So, since $N(a_4)\cap D$ is complete to $\{b_1,v_1,v_2\}$, $\chi(G[N(a_4)\cap D])\leq\omega(G)-3$. Also, if there are vertices, say $d,d'\in D\sm N(a_4)$, then $\{a_2,a_4,v_4,d,d'\}$ induces a $\overline{P_2+P_3}$; so $|D\sm N(a_4)|\leq 1$. Thus, $(D\sm N(a_4))\cup  \{b_1\}$ is a stable set. Then, by \ref{ATstset}, since $A_1\cup T\cup \{v_2,v_4\}$, $B_1\sm \{b_1\}$ and $A_2\cup \{v_3\}$ are stable sets, by \cref{AB1B2bipart}, $\chi(G[V(G)\sm (N(a_4)\cap D)])\leq 6$, and we conclude that $\chi(G)\leq 6+(\omega(G)-3) = \omega(G)+3$. This completes the proof of the theorem.
\end{proof}

\subsection{($P_2+P_3$, $\overline{P_2+P_3}$, $K_{2,3}$, banner)-free graphs that contain an $H_2$}

\begin{thm}\label{thm:5apple}
If  $G$ is a ($P_2+P_3$, $\overline{P_2+P_3}$, $K_{2}+ K_{3}$, co-banner)-free graph that contains an $\overline{H_2}$, then $\overline{G}$ is a good graph.
\end{thm}
\begin{proof} Let $G$ be a ($P_2+P_3$, $\overline{P_2+P_3}$, $K_{2}+ K_{3}$, co-banner)-free graph that contains an $\overline{H_2}$. We may consider an $\overline{H_2}$   with vertices and edges as in Figure~\ref{fig:partfig2}. Let $C:=\{v_1,v_2,v_3,v_4\}$.  We partition $V(G)\sm C$ as in Section~\ref{genprop},  and we use the properties in Section~\ref{genprop}.
Then, clearly, $b_1\in B_1$ and $b_2\in B_2$.  So, by \ref{AiB+2com}, we have $|B_3| \leq 1$ and $|B_4|\leq 1$. Moreover,  if there is a vertex, say $a_1\in A_1$, then, by \ref{ABst}, $a_1b_1\notin E(G)$, and by \ref{AiB+2com}, $a_1b_2\in E(G)$, and then $\{a_1,v_1,v_2,b_2,b_1\}$ induces a $\overline{P_2+P_3}$; so $A_1=\es$.  Likewise, $A_3=\es$. Hence $A=A_2\cup A_4$. Next if there are   vertices, say $x,x'\in X_2$, then, by \ref{Xstset}, $xx'\notin E(G)$, and then, by \ref{BXantcomp},  $\{b_1,b_2,x,v_4,x'\}$ induces a  $P_2+P_3$; so $|X_2|\leq 1$. Moreover, we have the following:
 \begin{claim}\label{5app-theta} \mbox{ $\theta(G[C\cup A\cup X_1\cup T])\leq \alpha(G)$.}
\end{claim}

\vspace{-0.2cm}
\no{\it Proof of \cref{5app-theta}}.~If there are adjacent vertices, say $a\in A_2$ and $a'\in A_4$, then, by \ref{ABst} and \ref{AiB+2com}, $\{a,a',b_1,v_2,b_2\}$ induces a $\overline{P_2+P_3}$; so $A_2$ is anticomplete to $A_4$, and hence, by \ref{ATstset}, $A_2\cup A_4\cup T\cup \{v_1,v_3\}$ is a stable set.  Moreover, by \ref{Xstset},   $X_1\cup \{v_2,v_4\}$ is also a stable set. So, $G[C\cup A\cup X_1\cup T]$ is a bipartite graph, and hence $\theta(G[C\cup A\cup X_1\cup T])\leq \alpha(G)$. This proves  \cref{5app-theta}. $\Diamond$

 We may assume that $\overline{G}$ has no universal or a pair of comparable vertices. So, to prove the theorem, by \cref{5app-theta},  we show that  either $v_2$ is a nice vertex in $\overline{G}$ or $\theta(G[B\cup D\cup X_2])\leq 3$.    Note that $|\overline{N}(v_2)|=|\{v_4\}\cup A_4\cup T \cup X_1|+|B_3|+|B_4|\leq |\{v_4\}\cup A_4\cup T \cup X_1|+2$. We consider two cases based on the set $A_4$.

\medskip
 \no{\bf Case~1}~{\it Suppose that $A_4\neq \es$.}

  \smallskip
 \no Let $a_4\in A_4$. Then, by \ref{AiB+2com},  $|B_1|= 1$ and $|B_2|= 1$; and so $B_1=\{b_1\}$ and $B_2=\{b_2\}$.  If $|X_1|\leq 1$, then since $\{v_4\}\cup A_4\cup T$ is a stable set (by \ref{ATstset}) which is anticomplete to $\{v_2\}$, we see that $|\overline{N}(v_2)|\leq (\alpha(G)-1)+3= \alpha(G)+2$, and hence $v_2$ is a nice vertex in $\overline{G}$. So, we may assume that $|X_1|\geq 2$; and let $x_1,x_1'\in X_1$.  Next:
 \begin{claim}\label{5app-Dclq}
 \mbox{$D$ is a clique. }
 \end{claim}

 \vspace{-0.2cm}
\no{\it Proof of \cref{5app-Dclq}}.~Suppose there are nonadjacent vertices in $D$, say $d_1$ and $d_2$. Then since $\{d_1,v_2,d_2,v_4,$ $a_4\}$ does not induce a $\overline{P_2+P_3}$, we have either $a_4d_1\notin E(G)$ or $a_4d_2\notin E(G)$. We may assume that $a_4d_1\notin E(G)$. If there is a vertex $x\in X_1$ such that $a_4x\in E(G)$, then, by \ref{XDcom}, $\{a_4,v_4,v_1,x,d_1\}$ induces a $\overline{P_2+P_3}$; so $\{a_4\}$ is anticomplete to $X_1$. Then, by \ref{AiB+2com} and \ref{BXantcomp}, $\{a_4,b_1,x_1,v_3,x_1'\}$ induces a $P_2+P_3$. This proves \cref{5app-Dclq}. $\Diamond$

Now, since $|X_2|\leq 1$, by \ref{XDcom} and \cref{5app-Dclq},  $D\cup X_2$ is a clique. Also, since $|B_i|\leq 1$, for each $i$, by \ref{AiB+2com}, $B_3\cup \{b_1\}$ and $B_4\cup \{b_2\}$ are cliques. So, $\theta(G[B\cup D\cup X_2])\leq 3$.

\medskip
\no{\bf Case~2}~{\it Suppose that $A_4=\es$.}

 \smallskip
 \no If one of $X_1$ and $T$ is empty, then, by \ref{ATstset} and \ref{Xstset}, $|\overline{N}(v_2)|\leq \alpha(G)+2$, and hence    $v_2$ is a nice vertex in $\overline{G}$. So, we may assume that  both $X_1$ and $T$ are nonempty; and let $x\in X_1$ and $t\in T$. Then:
\begin{claim}\label{5app-BT}
\mbox{$\{t\}$ is complete to $B$.}
\end{claim}

\vspace{-0.2cm}
\no{\it Proof of \cref{5app-BT}}. Suppose that there is a vertex, say $b\in B$ such that $bt\notin E(G)$. We may assume that $b\in B_1$. Recall that, by \ref{BXantcomp}, $B$ is anticomplete to $X$. Now, for $x\in X_1$, since $\{b,v_1,v_2,t,x\}$ does not induce a co-banner, $\{t\}$ is anticomplete to $X_1$. Likewise, $\{t\}$ is anticomplete to $X_2$. Next, if $t$ has a neighbor in some $B_i$, say $b'$, then $\{b',t,v_{i+2},v_{i+3},x\}$ induces a $P_2+P_3$; so $\{t\}$ is anticomplete to $B$. Since $\{t\}$ is anticomplete to $C\cup A\cup B\cup X$, and since $\overline{G}$ has no universal vertices, $t$ must have a neighbor in $D$, say $d$. Then, by \ref{abtnbBD}, $N_G(t)$ is a clique. Thus, $\overline{N}_G(d)\subseteq V(G)\sm (N(t)\cup \{t\})=\overline{N}_G(t)$, and so $N_{\overline{G}}(d)\subseteq N_{\overline{G}}(t)$. Hence, $d$ and $t$ are comparable vertices in $\overline{G}$, a contradiction.  This proves \cref{5app-BT}. $\Diamond$

Next:
\begin{claim}\label{5app-BD}
\mbox{$B_1=\{b_1\}$, $B_2=\{b_2\}$, and $B$ is complete to $D$.}
\end{claim}

\vspace{-0.2cm}
\no{\it Proof of \cref{5app-BD}}. By \cref{5app-BT} and by  \ref{A1A2onenbd}, we have $B_1=\{b_1\}$. Likewise, $B_2=\{b_2\}$. Next, if there are nonadjacent vertices, say $b\in B_i$ and $d\in D$, then, by \ref{abtnbBD}, $dt\notin E(G)$, and then, by \cref{5app-BT}, $\{b,t,d,v_{i+2},v_{i+3}\}$ induces a $K_2+K_3$; so $B$ is complete to $D$.  This proves \cref{5app-BD}. $\Diamond$

Now since $|B_3|\leq 1$, $|B_4|\leq 1$ and $|X_2|\leq 1$, by \cref{5app-BD},  by \ref{abtnbBD} and by \ref{AiB+2com}, clearly $B_1\cup B_3\cup D$, $B_2\cup B_4$ and $X_2$ are cliques. So, $\theta(G[B\cup D\cup X_2])\leq 3$.  This completes the proof of the theorem.
\end{proof}

\subsection{($P_2+P_3$, $\overline{P_2+P_3}$, $K_{2}+K_{3}$, co-banner, $\overline{H_2}$)-free graphs that contain an $\overline{H_3}$}

Let $G$ be a ($P_2+P_3$, $\overline{P_2+P_3}$, $K_{2}+K_{3}$, co-banner, $\overline{H_2}$)-free graph such  that $\overline{G}$ has no universal or a pair of comparable vertices. We may consider an $G$ contains an $\overline{H_3}$ with vertices and edges as in Figure~\ref{fig:partfig2}. Let $C:=\{v_1,v_2,v_3,v_4\}$.  We partition $V(G)\sm C$ as in Section~\ref{genprop},  and we use the properties in Section~\ref{genprop}.
 Clearly, $b^*\in B$ and $t^*\in T$. Also, since $b^*t^*\in E(G)$,  $B$ is not anticomplete to $T$.  Moreover, the graph $G$ has some more properties which we give in \cref{H3-Prop1,H3-Prop2,AiAi+2oneofemp,H3-AiAi+2ancomforali,H2-AiAi+2K22case,H2-AiAi+2K13case,H3-AiAi+2notanticom} below.

 \begin{lemma}\label{H3-Prop1}The following hold:

\vspace{-0.3cm}
\begin{lemmalist}
\item\label{BiBi+1} For   $i\in \{1,2,3,4\}$, $B_i\cup B_{i+1}$ is a stable set.
\item\label{BTcom} $B$ is complete to $T$.
\item\label{BTcard} For   $i\in \{1,2,3,4\}$, $|B_i|\leq 1$, and $|T|=1$.
\item\label{BDcom} $B$ is complete to $D$. Moreover, $D$ is a clique.
\item\label{theta-BCDT} $\theta(G[B\cup D\cup T])\leq 2$ and $\theta(G[B\cup C\cup D\cup T])\leq 3$.
\end{lemmalist}
\end{lemma}
\begin{proof} $(i)$:~By \ref{ABst}, it is enough to show that $B_i$ is anticomplete to $B_{i+1}$. Now if  there are  adjacent vertices, say $b\in B_i$ and $b'\in B_{i+1}$, then $C\cup \{b,b'\}$ induces an $\overline{H_2}$. $\diamond$

\no{$(ii)$:}~Suppose  there are nonadjacent vertices, say $b\in B$ and $t\in T$. We may assume that $b\in B_1$. Now, if $t$ has a neighbor in $X$, say $x$, then, by \ref{BXantcomp}, $\{t,x,v_1,v_2,b\}$ induces a co-banner, and if  $t$ has a neighbor in $B_2\cup B_4$, say $b'\in B_2$, then  by \cref{BiBi+1}, $bb'\notin E(G)$, and then $\{b',t,b,v_1,v_4\}$ induces a $P_2+P_3$. These contradictions together with \ref{ATstset} show that $\{t\}$ is anticomplete to $C\cup A\cup B_2\cup B_4\cup X$.  Since $\overline{G}$ has no universal vertices,   $t$ must have a neighbor in $B_1\cup B_3\cup D$, say $p$. By \ref{abtnbBD} and \ref{AiB+2com}, $N_G(t)$ is a clique. Thus, $\overline{N}_G(p)\subseteq V(G)\sm (N(t)\cup \{t\})=\overline{N}_G(t)$, and so $N_{\overline{G}}(p)\subseteq N_{\overline{G}}(t)$. Hence $p$ and $t$ are comparable vertices in $\overline{G}$, a contradiction. $\diamond$

\no{$(iii)$:}~By \ref{A1A2onenbd} and \cref{BTcom}, clearly $|B_i|\leq 1$. Next, if there are vertices, say $t,t'\in T$, then, by \cref{BTcom} and \ref{ATstset}, $\{v_3,v_4,t,b^*,t'\}$ induces a $P_2+P_3$; so $|T|\leq 1$. Since $T\neq \es$, we have $|T|=1$.
 $\diamond$

\no{$(iv)$:}~The proof of the first assertion follows from our assumption that $B, T\neq \es$, and is similar to the proof of \cref{5app-BD} in \cref{thm:5apple}, and we omit the details. Since $B\neq \es$, the second assertion follows from the first assertion and from \ref{abtnbBD}. $\diamond$

\no{$(v)$:}~We use \ref{AiB+2com}, \cref{BTcard}, and \cref{BDcom}. Since $B_1\cup B_3\cup T$ and $B_2\cup B_4\cup D$ are cliques,  clearly $\theta(G[B\cup D\cup T])\leq 2$. Also, since   $B_1\cup B_3\cup T$, $B_2\cup D\cup \{v_2,v_3\}$ and $B_4\cup \{v_1,v_4\}$ are cliques, we have $\theta(G[B\cup C\cup D\cup T])\leq 3$.
\end{proof}

\begin{lemma}\label{H3-Prop2}The following hold:

\vspace{-0.3cm}
\begin{lemmalist}
\item\label{XTcom} $X$ is complete to $T$.
\item\label{AiXjmatchcom} Let $j,k\in \{1,2\}$ and $j\neq k$. Then for any $a\in A_j\cup A_{j+2}$, we have $|X_k\sm N(a)|\leq 1$, and   	for any $x\in X_k$, we have $|A_j\sm N(x)|\leq 1$ and $|A_{j+2}\sm N(x)|\leq 1$.
\end{lemmalist}
\end{lemma}
 \begin{proof}
$(i)$:~If there are nonadjacent vertices, say  $x\in X$ and $t\in T$, then, for any $b\in B$, say $b\in B_i$, by \cref{BTcom} and \ref{BXantcomp},  $\{b,t,x,v_{i+2},v_{i+3}\}$ induces a $P_2+P_3$. $\diamond$

\no{$(ii)$:}~We prove for $j=1$. By symmetry, we may assume that $a\in A_1$. If there are vertices, say $x_2,x_2'\in X_2\sm N(a)$, then, by \ref{ATstset} and \cref{XTcom}, $\{a,v_1,x_2,t^*,x_2'\}$ induces a $P_2+P_3$. Next, if  there are  vertices, say $a_1,a_1'\in A_1\sm N(x)$, then,   by \ref{ATstset} and \cref{XTcom}, $\{t^*,x,a_1,v_1,a_1'\}$ induces a $P_2+P_3$.

\end{proof}

\begin{lemma}\label{AiAi+2oneofemp}
	If  one of $A_i$, $A_{i+2}$ is empty,  for each $i\in \{1,2,3,4\}$,  then  $\overline{G}$ is a good graph.
\end{lemma}
\begin{proof}
We may assume that $A_3\cup A_4=\es$. Also:

\vspace{-0.2cm}
 \begin{claim1}\label{H2A3A4es-AX}
 We may assume that $A_1\cup X_1$ and $A_2\cup X_2$ are not stable sets.
\end{claim1}

\vspace{-0.25cm}
\no{\it Proof of \ref{H2A3A4es-AX}}.~If  $A_1\cup X_1$ is a stable set, then, by \ref{ABst}, \ref{BXantcomp} and \cref{BTcard}, $|\overline{N}(v_2)|=|A_1\cup B_4\cup X_1|+|B_3|+|T|+|\{v_4\}|\leq (\alpha(G)-1)+3=\alpha(G)+2$; so $v_2$ is a nice vertex in $\overline{G}$, and we are done. So we may assume that $A_1\cup X_1$ is not a stable set. Likewise,  $A_2\cup X_2$ is not a stable set.  $\Diamond$

\vspace{-0.2cm}
 \begin{claim1}\label{H2A3A4es-A1A2card2}
 We may assume that  $|A_1|\geq 2$ and   $|A_2|\geq 2$.
\end{claim1}

\vspace{-0.25cm}
 \no{\it Proof of \ref{H2A3A4es-A1A2card2}}.~If $|A_1|\leq 1$, then since $B_3\cup B_4\cup X_1$ is a stable set (by \ref{BXantcomp} and \cref{BiBi+1}),  by \cref{BTcard}, $|\overline{N}(v_2)|=|B_3\cup B_4\cup X_1|+|A_1|+|T|+|\{v_{4}\}|\leq (\alpha(G)-1)+3=\alpha(G)+2$; so $v_2$ is a nice vertex in $\overline{G}$, and we are done.  Hence we may assume that $|A_1|\geq 2$. Likewise, $|A_2|\geq 2$.
 $\Diamond$

\vspace{-0.2cm}
 \begin{claim1} \label{H2A3A4es-X1X2card2}
 We may assume that $|X_1|\geq 2$ and  $|X_2|\geq 2$.
\end{claim1}

\vspace{-0.25cm}
 \no{\it Proof of \ref{H2A3A4es-X1X2card2}}.~If $|X_1|\leq 1$, then since $A_1\cup T\cup \{v_4\}$ is a stable set (by \ref{ATstset}),  by \cref{BTcard}, $|\overline{N}(v_2)|=|A_1\cup T \cup \{v_4\}|+ |B_3\cup B_4|+|X_1|\leq (\alpha(G)-1)+3=\alpha(G)+2$; so again $v_2$ is a nice vertex in $\overline{G}$, and we are done.  Hence we may assume that $|X_1|\geq 2$. Likewise, $|X_2|\geq 2$.
 $\Diamond$

  By \ref{H2A3A4es-AX}, there are adjacent vertices, say $a_1\in A_1$ and $x_1\in X_1$. Moreover, we claim the following:

\vspace{-0.2cm}
 \begin{claim1}\label{H2A3A4es-A1A2card3}
  $|A_1|\leq 3$, and $|A_2|\leq 3$.
\end{claim1}

\vspace{-0.25cm}
 \no{\it Proof of \ref{H2A3A4es-A1A2card3}}.~Suppose that $|A_1|\geq 4$. Then since $|X_2|\geq 2$, by \cref{AiXjmatchcom}, there is a vertex in $X_2$, say $x_2$, such that $a_1x_2\in E(G)$. Again, by \cref{AiXjmatchcom} and by the pigeonhole principle, there are vertices, say $a_1',a_1''\in A_1\sm \{a_1\}$ such that $a_1'x_2,a_1''x_2\in E(G)$. Now, since $\{x_1,x_2,v_1,v_2,a_1\}$ does not induce a $\overline{P_2+P_3}$, we have $x_1x_2\notin E(G)$.  Then, since $\{a_1',x_1,a_1,x_2,v_1\}$ does not induce a $\overline{P_2+P_3}$, $a_1'x_1\notin E(G)$. Likewise $a_1''x_1\notin E(G)$. But now $\{v_3,x_1,a_1',x_2,a_1''\}$ induces a $P_2+P_3$. So, we have $|A_1|\leq 3$. Likewise, $|A_2|\leq 3$.
$\Diamond$

\vspace{-0.2cm}
 \begin{claim1}\label{H2A3A4es-X1X2card3}
 $|X_1|\leq 3$, and $|X_2|\leq 3$.
\end{claim1}

\vspace{-0.25cm}
\no{\it Proof of \ref{H2A3A4es-X1X2card3}}.~The proof is similar to the proof of \ref{H2A3A4es-A1A2card3}, and we omit the details. 
$\Diamond$

\vspace{-0.2cm}
 \begin{claim1} \label{H2A3A4es-theta}
  $\theta(G[A_1\cup X_2])\leq 3$ and $\theta(G[A_2\cup X_1])\leq 3$.
\end{claim1}

\vspace{-0.25cm}
\no{\it Proof of \ref{H2A3A4es-theta}}.~This follows from \ref{ATstset}, \ref{Xstset}, \cref{AiXjmatchcom}, \ref{H2A3A4es-A1A2card3} and by \ref{H2A3A4es-X1X2card3}.
$\Diamond$

 \medskip
Now, by \ref{H2A3A4es-theta} and \cref{theta-BCDT}, we conclude that $\theta(G)\leq 9$. Also, by \ref{H2A3A4es-A1A2card2} and by \ref{ATstset}, since $A_1\cup T\cup \{v_2,v_4\}$ is a stable set, we have $\alpha(G)\geq 5$. If $\alpha(G)\geq 6$, then $\theta(G)\leq \alpha(G)+3$, and we are done. So, we may assume that $\alpha(G)=5$. Since  $A_1\cup T\cup \{v_2,v_4\}$ is a stable set (by \ref{ATstset}), by \ref{H2A3A4es-A1A2card2}, $|A_1|=2$. Likewise, $|A_2|=2$.
Now, if $|X_1|=2$, then, by \cref{AiXjmatchcom}, $\theta(G[A_2\cup X_1])\leq 2$, and so by \ref{H2A3A4es-theta}  and \cref{theta-BCDT}, $\theta(G)\leq 8=\alpha(G)+3$. So, we may assume that $|X_1|=3$ and $|X_2|=3$. Since $X_1\cup B_1\cup B_2\cup \{v_4\}$ is a stable set, one of $B_2, B_4$ is empty. We may assume that $B_2=\es$. Then as in \ref{H2A3A4es-theta}, we have  $\theta(G[A_2\cup (X_1\sm \{x_1\})])\leq 2$, and $\theta(G[A_1\cup X_2\cup \{v_2\}])\leq 3$. Also,  $D\cup \{v_3,x_1\}$ is a clique (by \cref{BDcom} and \ref{XDcom}),  $B_4\cup\{v_1,v_4\}$ is a clique (by  \cref{BTcard}), and $B_1\cup B_3\cup T$ is a clique (by \ref{AiB+2com}, \cref{BTcom} and \cref{BTcard}). So, we conclude that $\theta(G)\leq 8 =\alpha(G)+3$. This completes the proof.
 \end{proof}

 \begin{lemma}\label{H3-AiAi+2ancomforali}
If  $A_i$ is anticomplete to $A_{i+2}$, for each $i\in \{1,2,3,4\}$, then   $\overline{G}$ is a good graph.
\end{lemma}
\begin{proof}
 First we observe the following:

 \vspace{-0.2cm}
 \begin{claim1}\label{Ai123}  If there are vertices, say $p\in A_{i}$,  $q\in A_{i+1}$, and $r\in A_{i+2}$, then
 $pq, qr\in E(G)$ or $pq, qr\notin E(G)$.
 \end{claim1}

 \vspace{-0.25cm}
 \no{\it Proof of \ref{Ai123}}.~If $pq\in E(G)$, and $qr\notin E(G)$ (say), then $\{r,v_{i+2},v_i,p,q\}$ induces a $P_2+P_3$.
 $\Diamond$

By \cref{AiAi+2oneofemp}, we may assume that $A_1,A_3\neq \es$. Then, by \ref{a1b1nonnbd}, we let $A_1:=\{a_1\}$ and  $A_3:=\{a_3\}$.
First suppose that $\{a_1\}$ is complete to $A_2\cup A_4$.  Then, by \ref{A1A2onenbd}, we have $|A_2|\leq 1$ and $|A_4|\leq 1$.
Also, by \ref{Ai123},  $\{a_3\}$ is complete to $A_4$. So, by \ref{AiB+2com},  \cref{BTcom} and \cref{BTcard},  $A_2\cup B_3\cup \{a_1\}$,  $A_4\cup B_1\cup \{a_3\}$, and $B_2\cup B_4\cup T$ are cliques. Also, by \ref{Xstset}, \ref{XDcom} and \cref{BDcom}, we conclude that $G[X_1\cup X_2\cup C\cup D]$ is a perfect graph, as it is a join of a bipartite graph and a complete graph. Thus, $\theta(G)\leq \alpha(G)+3$, and we are done. So we may assume that $\{a_1\}$ is not complete to $A_2\cup A_4$, and let $a_2\in A_2$ be such that $a_1a_2\notin E(G)$. So, by \ref{Ai123}, $\{a_1,a_2,a_3\}$ is  a stable set. We consider two cases based on the set $A_4$.

\medskip
\no{\bf Case~1} $A_4\neq \es$.
 	
 \smallskip
 By \ref{a1b1nonnbd}, $A_2\sm \{a_2\}=\es$ and $|A_4|=1$. Let $A_4:=\{a_4\}$. So, by \ref{Ai123}, $A = \{a_1,a_2,a_3,a_4\}$ is  a stable set. Next, we claim that:

\vspace{-0.2cm}
 \begin{claim1}\label{theta-XA}
$|X_1|\leq 3$ and $|X_2|\leq 3$. Hence, $\theta(G[X_1\cup \{a_1,a_3\}])\leq 3$ and $\theta(G[X_2\cup \{a_2,a_4\}])\leq 3$.
\end{claim1}

\vspace{-0.25cm}
 \no{\it Proof of \ref{theta-XA}}.~Suppose $|X_1|\geq 4$.  Then, by \cref{AiXjmatchcom}, there are vertices $p,q,r\in X_1$ such that $\{a_2\}$ is complete to $\{p,q,r\}$. Then, by \ref{AXonenbd}, we may assume that $pa_3,qa_3\notin E(G)$. But, now \ref{Xstset},  \ref{BXantcomp}, \ref{ABst} and by \ref{AiB+2com},  $\{b^*,a_3,a_2,p,q\}$ induces $P_2+P_3$. So, $|X_1|\leq 3$. Likewise, $|X_2|\leq 3$. This proves the first assertion. The second assertion follows from the first and from \cref{AiXjmatchcom}.
 $\Diamond$

 By \cref{theta-BCDT} and \ref{theta-XA},  we have $\theta(G)\leq 9$. If $\alpha(G)\geq 6$, then $\theta(G)\leq \alpha(G)+3$. So we may assume that $\alpha(G)\leq 5$. Since $T\cup \{a_1,a_3,v_2,v_4\}$ is a stable set of size $5$, we have $\alpha(G)= 5$.

 \vspace{-0.2cm}
 \begin{claim1}\label{H3AiAi+2-theta}
		$\theta(G[X_1\cup \{a_1,a_3,v_1,v_3\}])\leq 3$. Likewise,  $\theta(G[X_2\cup \{a_2,a_4,v_2,v_4\}])\leq 3$.
	\end{claim1}

\vspace{-0.25cm}
 \no{\it Proof of \ref{H3AiAi+2-theta}}.~Since $X_1\cup \{v_2,v_4\}$ is a stable set (by \ref{Xstset}), and since $\alpha(G)=5$, $|X_1|\leq 3$.
If $|X_1|\leq 1$, then $\{a_1,v_1\}, \{a_3,v_3\}, X_1$ are  cliques, and $\theta(G[X_1\cup \{a_1,a_3,v_1,v_3\}])\leq 3$; so we may assume that $|X_1|\geq 2$. Let $x_1,x_1'\in X_1$. Since $\{a_1,a_3,v_1,v_3,x_1,x_1'\}$ is a stable set of size $6$, we may assume that $a_1x_1\in E(G)$. Then since $\{a_1,v_1,x_1,v_3,a_3\}$ does not induce a co-banner, $a_3x_1\in E(G)$. So, by \ref{AXonenbd}, $\{a_1,a_3\}$ is anticomplete to $X_1\sm \{x_1\}$. Since $(X_1\sm \{x_1\})\cup \{a_1,a_3,v_2,v_4\}$ is a stable set of size at most $5$, $|X_1\sm \{x_1\}|=1$, and so $X_1\sm x_1=\{x_1'\}$. Then since  $\{a_1,v_1,x_1\}$, $\{a_3,v_3\}$, and $\{x_1'\}$ are cliques, $\theta(G[X_1\cup \{a_1,a_3,v_1,v_3\}])\leq 3$. This proves \ref{H3AiAi+2-theta}. $\Diamond$

So, by \ref{H3AiAi+2-theta} and \cref{theta-BCDT}, we have    $\theta(G)\leq 8=\alpha(G)+3$.


\medskip
\no{\bf Case~2} $A_4=\es$.
 	
 \smallskip
Let $A_2':=\{a\in A_2\mid aa_1\in E(G)\}$; so $a_2\in A_2\sm A_2'$.  By \ref{A1A2onenbd}, $|A_2'|\leq 1$. Next, we claim the following:

\vspace{-0.2cm}
 \begin{claim1}\label{H2AiAi+2-X}
We may assume that $|X_1|\geq 2$.
\end{claim1}

\vspace{-0.25cm}
 \no{\it Proof of \ref{H2AiAi+2-X}}.~If $|X_1|\leq 1$, since  $T\cup \{a_1,a_3,v_4\}$ is a stable set (by \ref{ATstset}), by \cref{BTcard}, $|\overline{N}(v_2)|=|X_1|+|T\cup \{a_1,a_3,v_4\}|+|B_3\cup B_4| \leq \alpha(G)+2$; so  $v_2$ is a nice vertex in $\overline{G}$, and we are done.  Hence we may assume that $|X_1|\geq 2$. $\Diamond$

 \vspace{-0.2cm}
 \begin{claim1}\label{H2AiAi+2-cardalp}
  $|X_1\cup \{a_1,a_3,v_4\}|\leq \alpha(G)$.
  \end{claim1}

 \vspace{-0.25cm}
 \no{\it Proof of \ref{H2AiAi+2-cardalp}}.~If $X_1\cup \{a_1,a_3,v_4\}$ is a stable set, we are done. So, we may assume that, by \ref{Xstset}, there is a vertex, say $x\in X_1$ such that $a_1x\in E(G)$. Then since  $\{a_1,a_3,v_1,x,v_3\}$ does not induce a co-banner, $a_3x\in E(G)$. Also, by \ref{AXonenbd}, $\{a_1\}$ is anticomplete to $X_1\sm \{x\}$, and $\{a_3\}$ is anticomplete to $X_1\sm \{x\}$.  So, by \ref{Xstset}, $(X_1\sm \{x\})\cup \{a_1,a_3,v_4\}$ is a stable set, and since $\{v_2\}$ is anticomplete to $(X_1\sm \{x\})\cup \{a_1,a_3,v_4\}$,   $|(X_1\sm \{x\})\cup \{a_1,a_3,v_4\}|\leq \alpha(G)-1$, and hence $|X_1\cup \{a_1,a_3,v_4\}|\leq \alpha(G)$.
$\Diamond$

\vspace{-0.2cm}
 \begin{claim1}\label{H2AiAi+2-B}
We may assume that $B_3,B_4\neq \es$.
\end{claim1}

\vspace{-0.25cm}
 \no{\it Proof of \ref{H2AiAi+2-B}}.~If $|B_3\cup B_4|\leq 1$, then, by \ref{H2AiAi+2-cardalp} and \cref{BTcard}, $|\overline{N}(v_2)|=|X_1\cup \{a_1,a_3,v_4\}|+|B_3\cup B_4|+|T| \leq \alpha(G)+2$; so  $v_2$ is a nice vertex in $\overline{G}$, and we are done.
So, by \cref{BTcard}, we conclude that $B_3,B_4\neq \es$. $\Diamond$

 \vspace{-0.2cm}
 \begin{claim1}\label{H3AiAi+2ancomX1A2clqcov}
 	$\theta(G[(A_2\sm A_2')\cup X_1])\leq \alpha(G)-3$.
\end{claim1}

\vspace{-0.25cm}
 \no{\it Proof of \ref{H3AiAi+2ancomX1A2clqcov}}.~ First, since  $X$ is anticomplete to $B$ (by \ref{BXantcomp}),   by \cref{BiBi+1} and \ref{Xstset}, $X_1\cup B_3\cup B_4 \cup \{v_2\}$ is a stable set, and so  by \ref{H2AiAi+2-B}, $|X_1|\leq \alpha(G)-3$.
 Now, if $A_2\sm A_2' =\{a_2\}$, then, by \ref{H2AiAi+2-X} and \cref{AiXjmatchcom}, $\{a_2\}$ is not anticomplete to $X_1$, and hence $\theta(G[(A_2\sm A_2')\cup X_1])\leq |X_1|\leq \alpha(G)-3$. So, we may assume that $|A_2\sm A_2'|\geq 2$.

For integers $r\geq 2$ and $k\geq 2$, let $X_1:=\{p_1,p_2,\ldots p_r\}$ and $A_2\sm A_2':=\{q_1,q_2,\ldots, q_k\}$.
If $r\geq k$. Then, by \cref{AiXjmatchcom}, we may assume that   $p_{j}q_{j}\in E(G)$, where $j\in\{1,2,\ldots,k\}$, and then  $\theta(G[(A_2\sm A_2')\cup X_1])\leq |X_1|\leq \alpha(G)-3$; so suppose that $r<k$. Then, again by \cref{AiXjmatchcom}, we may assume that   $p_{j}q_{j}\in E(G)$, where $j\in\{1,2,\ldots,r\}$, and hence $\theta(G[(A_2\sm A_2')\cup X_1])\leq k=|A_2\sm A_2'|$. Now, since  $\{a_3\}$ is anticomplete to $A_2\sm A_2'$ (by \ref{Ai123}), $(A_2\sm A_2')\cup \{a_1,a_3,v_4\}$ is a stable set, and so $|A_2\sm A_2'|\leq \alpha(G)-3$.  Thus, we conclude that $\theta(G[(A_2\sm A_2')\cup X_1])\leq \alpha(G)-3$.
$\Diamond$

\vspace{-0.2cm}
 \begin{claim1}\label{H3AiAi+2ancomX2A1A3clq}
	$\theta(G[A_2'\cup X_2\cup\{a_1,a_3\}])\leq 3$.
\end{claim1}

\vspace{-0.25cm}
 \no{\it Proof of \ref{H3AiAi+2ancomX2A1A3clq}}.~To prove the claim, we partition $X_2$ as follows: $X_2':=\{x\in X_2\mid xa_2\in E(G)\}$, $X_2'':=\{x\in X_2\mid  xa_1\in E(G), xa_2\notin E(G)\}$  and $X_2''':=\{x\in X_2\mid xa_1,xa_2\notin E(G)\}$. By \ref{AXonenbd}, $|X_2'|\leq 1$, and by \cref{AiXjmatchcom}, $|X_2'''|\leq 1$. By \ref{a1a2x}, $A_2'\cup X_2'''$ is a clique.  Since for any $x\in X_2'$, $\{x,a_2,v_2,v_3,a_3\}$ does not induce a co-bannner, $X_2'\cup \{a_3\}$ is a clique. Also, if there are vertices, say  $x,x'\in X_2''$, then, for any $b\in B_4$, by \ref{ABst} and \ref{AiB+2com}, $\{a_2,b,x,a_1,x'\}$  induces a $P_2+P_3$; so $|X_2''|\leq 1$, and hence $X_2''\cup \{a_1\}$ is a clique.
So we conclude that $\theta(G[A_2'\cup X_2\cup \{a_1,a_3\}])\leq 3$.
$\Diamond$

\smallskip	
Now, by \ref{H3AiAi+2ancomX1A2clqcov} and \ref{H3AiAi+2ancomX2A1A3clq}, we have $\theta(G[A\cup X])\leq \alpha(G)$, and so,  by \cref{theta-BCDT}, we have $\theta(G)\leq \alpha(G)+3$.
 This proves \cref{H3-AiAi+2ancomforali}.
\end{proof}

\begin{lemma}\label{H2-AiAi+2K22case}
		If $A_i\cup A_{i+2}$ induces a $K_{2,2}$, for some $i\in \{1,2,3,4\}$, then $\overline{G}$ is a good graph.
	\end{lemma}
\begin{proof}	We may assume that $i=1$. By \ref{ATstset}, we may assume that there are vertices, say $p_1,p_2\in A_1$ and $q_1, q_2\in A_3$ such that $\{p_1,p_2\}$ is complete to $\{q_1,q_2\}$. Then we claim the following.

\vspace{-0.2cm}
 \begin{claim1}\label{H2-K22-A2A4}
 $A_2\cup A_4=\es$.
\end{claim1}

\vspace{-0.25cm}
 \no{\it Proof of \ref{H2-K22-A2A4}}.~Suppose, up to symmetry, there is a vertex, say $a\in A_2$.  By \ref{A1A2onenbd}, we may assume that $ap_1,aq_1\notin E(G)$. Then since $\{a,v_2,p_1,q_1,p_2\}$ does not induce a $P_2+P_3$, $ap_2\in E(G)$. Likewise, $aq_2\in E(G)$. But then $\{a,p_2,q_2,v_3,v_4\}$ induces a co-banner.
 $\Diamond$
	
 \vspace{-0.2cm}
 \begin{claim1}\label{H2-K22-X2}
$X_2=\es$.
\end{claim1}

\vspace{-0.25cm}
 \no{\it Proof of \ref{H2-K22-X2}}.~Suppose there is a vertex, say $x\in X_2$. Then, by \cref{AiXjmatchcom}, we may assume $xp_1,xq_1\in E(G)$. Also, by \ref{a1a3cmnnbda2}:(c), we may assume that $xp_2\in E(G)$. Then $\{p_1,v_1,p_2,q_1,x\}$ induces a $\overline{P_2+P_3}$.
  $\Diamond$

\smallskip	
	Now, by \ref{H2-K22-A2A4} and \ref{H2-K22-X2}, $|\overline{N}(v_1)|=|A_3\cup T\cup B_2\cup B_3\cup \{v_3\}|$. Since $A_3\cup T\cup \{v_2,v_4\}$ is a stable set (by \ref{ATstset}),    $|A_3\cup T|\leq \alpha(G)-2$. So, by \cref{BTcard},   $|\overline{N}(v_1)|=|A_3\cup T|+|B_2\cup B_3|+|\{v_3\}|\leq (\alpha(G)-2)+3< \alpha(G)+2$. This implies that $v_1$ is a nice vertex in $\overline{G}$, and hence $\overline{G}$ is a good graph. This proves \cref{H2-AiAi+2K22case}. \end{proof}
	
\begin{lemma}\label{H2-AiAi+2K13case}
If $A_i\cup A_{i+2}$ induces a $K_{1,3}$, for some $i\in \{1,2,3,4\}$, then  $\overline{G}$ is a good graph.
\end{lemma}
\begin{proof}	We may assume that $i=1$.  By \ref{ATstset}, we may assume that there are vertices, say $p_1\in A_1$ and $q_1, q_2, q_3\in A_3$ such that $\{p_1\}$ is complete to $\{q_1,q_2,q_3\}$.  By \cref{H2-AiAi+2K22case}, we may assume that
$G[A_1\cup A_3]$ is $K_{2,2}$-free.
Then we claim the following.

\vspace{-0.2cm}
 \begin{claim1}\label{H2-K13-A1}
$A_1\sm \{p_1\}=\es$.
\end{claim1}

\vspace{-0.25cm}
 \no{\it Proof of \ref{H2-K13-A1}}.~If there is a vertex, say $p_2\in A_1\sm \{p_1\}$, then, by \ref{a1b1nonnbd}, we may assume that $p_2q_1, p_2q_2\in E(G)$, and then $\{p_1,p_2,q_1,q_2\}$ induces a $K_{2,2}$ in $G[A_1\cup A_3]$, a contradiction.
 $\Diamond$
	
 \vspace{-0.2cm}
 \begin{claim1}\label{H2-K13-A124}
$\{p_1\}$ is complete to $A_2$. Likewise, $\{p_1\}$ is complete to $A_4$.
\end{claim1}

\vspace{-0.25cm}
 \no{\it Proof of \ref{H2-K13-A124}}.~If there is a vertex, say $a\in A_2$ such that $ap_1\notin E(G)$, then, by \ref{A1A2onenbd}, we may assume that $aq_1,aq_2\notin E(G)$, and then $\{a,v_2,q_1,p_1,q_2\}$ induces a $P_2+P_3$.
 $\Diamond$

 \vspace{-0.2cm}
 \begin{claim1}\label{H2-K13-X2}
 $|X_2|\leq 1$.
\end{claim1}

\vspace{-0.25cm}
 \no{\it Proof of \ref{H2-K13-X2}}.~If there are vertices, say  $x_2,x_2'\in X_2$, then, by \cref{AiXjmatchcom}, we may assume that $p_1x_2\in E(G)$ and $q_1x_2,q_2x_2\in E(G)$, and then $\{p_1,q_1,v_3,q_2,x_2\}$ induces a $\overline{P_2+P_3}$.
 $\Diamond$

		\smallskip

	  By \ref{H2-K13-A124} and by \ref{A1A2onenbd}, $|A_2|\leq 1$ and $|A_4|\leq 1$. So, by \ref{H2-K13-A1}, \ref{H2-K13-A124}, \ref{AiB+2com} and by \cref{BTcard}, $A_1\cup A_2\cup B_3$ and $A_4\cup B_1$  are cliques.
	Also, by \cref{BTcard} and \cref{BDcom}, $B_2\cup D\cup \{v_2,v_3\}$ and $B_4\cup \{v_1,v_4\}$ are cliques. So, by \ref{H2-K13-X2}, we conclude that $\theta(G-(A_3\cup T\cup X_1)) \leq 5$.
	Moreover, by \ref{ATstset} and \ref{Xstset}, $A_3\cup T\cup X_1$ induces a bipartite graph, and is anticomplete to $\{v_2,v_4\}$; so $\theta(G[A_3\cup T\cup X_1])\leq  \alpha(G)-2$. Hence $\theta(G)\leq\theta(G[A_3\cup T\cup X_1])+5\leq \alpha(G)+3$. This proves \cref{H2-AiAi+2K13case}.
\end{proof}

\begin{lemma}\label{H3-AiAi+2notanticom}
	If $A_i$ is not anticomplete to $A_{i+2}$, for some $i\in \{1,2,3,4\}$, then $\overline{G}$ is a good graph.
\end{lemma}
\begin{proof}
We may assume that $i=1$. Let $a_1\in A_1$ and $a_3\in A_3$ be such that $a_1a_3\in E(G)$.
We define $A_1':=\{a\in A_1\sm \{a_1\}\mid aa_3\in E(G)\}$, $A_1'':=\{a\in A_1\sm \{a_1\}\mid aa_3\notin E(G)\}$, $A_3':=\{a\in A_3\sm \{a_3\}\mid aa_1\in E(G)\}$ and $A_3'':=\{a\in A_3\sm \{a_3\}\mid aa_1\notin E(G)\}$. Then $A_1\sm \{a_1\}= A_1'\cup A_1''$ and $A_3\sm \{a_3\}= A_3'\cup A_3''$.
By \cref{H2-AiAi+2K13case}, $|A_1'|\leq 1$ and $|A_3'|\leq 1$. By \ref{a1b1nonnbd}, $|A_1''|\leq 1$ and $|A_3''|\leq 1$.  Now we claim the following:

\vspace{-0.2cm}
 \begin{claim1}\label{H3A1A3notancomclq}
	$\theta(G[A_1\cup A_3\cup X_2])\leq 3$.
\end{claim1}
	
\vspace{-0.25cm}
 \no{\it Proof of \ref{H3A1A3notancomclq}}.~To prove the claim, by \ref{a1a3cmnnbda2}:\ref{X2DA1A3}, first we partition $X_2$ as follows: $X_2':=\{x\in X_2\mid a_1x,a_3x\in E(G)\}$, $X_2'':=\{x\in X_2\mid a_1x\in E(G),a_3x\notin E(G)\}$ and $X_2''':=\{x\in X_2\mid a_3x\in E(G),a_1x\notin E(G)\}$. Now if there are vertices, say $x,x'\in X_2'$, then $\{a_1,x,v_2,x',a_3\}$ induces a $\overline{P_2+P_3}$; so $|X_2'|\leq 1$, and hence $X_2'\cup\{a_1,a_3\}$ is a clique. Also, by \cref{AiXjmatchcom}, $|X_2''|\leq 1$ and $|X_2'''|\leq 1$. Moreover,  for any $x\in X_2''$ and $a\in A_3''$, since $\{a,v_3,x,a_1,v_1\}$ does not induce a $P_2+P_3$, $A_3''$ is complete to $X_2''$. So by \ref{a1a3cmnnbda2}:\ref{X2DA1A3}, $A_1'\cup A_3''\cup X_2''$ is a clique. Likewise, $A_1''\cup A_3'\cup X_2'''$ is also a clique. Hence $\theta(G[A_1\cup A_3\cup X_2])\leq 3$. $\Diamond$

\vspace{-0.2cm}
 \begin{claim1}\label{H3A2A4ancomclq}
	 $\theta(G[A_2\cup A_4\cup X_1])\leq 3$.
\end{claim1}

\vspace{-0.25cm}
 \no{\it Proof of \ref{H3A2A4ancomclq}}.~If  $A_2$ is not anticomplete to $A_4$, then as in \ref{H3A1A3notancomclq}, we have $\theta(G[A_2\cup A_4\cup X_1])\leq 3$. So we may assume that $A_2$ is anticomplete to $A_4$.   We partition $X_1$  as follows: $X_1':=\{x\in X_1\mid a_1x\in E(G)\}$, $X_1'':=\{x\in X_1\mid a_1x\notin E(G),a_3x\in E(G)\}$ and $X_1''':=\{x\in X_1\mid a_1x,a_3x\notin E(G)\}$. By \ref{AXonenbd}, $|X_1'|\leq 1$ and $|X_1''|\leq 1$.
 If there are vertices $x,x'\in X_1'''$, by \cref{XTcom},    $\{a_1,a_3,x,t^*,x'\}$ induces a $P_2+P_3$; so $|X_1'''|\leq 1$. Hence $|X_1|\leq 3$. If $A_2,A_4\neq \es$, then, by \ref{a1b1nonnbd}, $|A_2|\leq 1$ and $|A_4|\leq 1$, and so, by \cref{AiXjmatchcom}, $\theta(G[A_2\cup A_4\cup X_1])\leq 3$. If $A_4=\es$ (up to symmetry), then, by \ref{a1a3cmnnbda2}:(a) and \ref{A1A2onenbd}, $|A_2|\leq 3$, and again by \cref{AiXjmatchcom}, $\theta(G[A_2\cup A_4\cup X_1])\leq 3$.  $\Diamond$

Now by \cref{theta-BCDT}, \ref{H3A1A3notancomclq} and  \ref{H3A2A4ancomclq}, $\theta(G)\leq 9$. If $\alpha(G)\geq 6$, we have $\theta(G)\leq \alpha(G)+3$. So it is enough to prove the lemma for $\alpha(G)\leq 5$ . Since $T\cup \{a_1,v_2,v_4\}$ is a stable set of size $4$ (by \ref{ATstset} and since $T\neq \es$), $\alpha(G)\geq 4$. Recall that, by \ref{Xstset},   $X_1\cup \{v_2,v_4\}$ and $X_2\cup\{v_1,v_3\}$ are stable sets, and
  by \ref{ATstset}, $A_i\cup T\cup \{v_{i+1},v_{i-1}\}$ is a stable set, for each $i$.
Then:

\vspace{-0.2cm}
 \begin{claim1}\label{H3alphafour}
	If $\alpha(G)= 4$, then $\theta(G)\leq 7$.
\end{claim1}

\vspace{-0.25cm}
 \no{\it Proof of \ref{H3alphafour}}.~Since $\alpha(G)= 4$, $|X_1|\leq 2$, $|X_2|\leq 2$, and  $|A_i|\leq 1$ for each $i$. For $j\in \{1,2\}$, if there are nonadjacent vertices $a\in A_j$ and $a'\in A_{j+2}$, by \ref{ATstset}, $T\cup \{a,a',v_{j+1},v_{j-1}\}$ is a stable set of size of  at least $5$, $A_j$ is complete to $A_{j+2}$. Then, by \ref{a1a3cmnnbda2}:\ref{X2DA1A3}, $\theta(G[A_1\cup A_3\cup X_2])\leq 2$ and $\theta(G[A_2\cup A_4\cup X_1])\leq 2$. So, by \cref{theta-BCDT}, we conclude that $\theta(G)\leq 7$.   $\Diamond$

By \ref{H3alphafour}, we may assume that $\alpha(G)=5$. As earlier, $|X_1|\leq 3, |X_2|\leq 3$, and $|A_i|\leq 2$ for each $i$. Hence, $|V(G)|\leq 23$ and $\alpha(G)=5$. So, as in \cref{AiAi+2oneofemp} or as in \cref{H3-AiAi+2ancomforali}:~Case~1, it is not hard to verify that $\theta(G)\leq 8$, and we omit the details.
\end{proof}

Now, we are in a position to prove the main result of this section, and is given below.

\begin{thm}\label{thm:caseH3}
 If $G$ is a ($P_2+P_3$, $\overline{P_2+P_3}$, $K_{2}+K_{3}$, co-banner, $\overline{H_2}$)-free graph that contains an $\overline{H_3}$, then $\overline{G}$ is a good graph.
\end{thm}
\begin{proof}
Let $G$ be a ($P_2+P_3$, $\overline{P_2+P_3}$, $K_{2}+K_{3}$, co-banner, $\overline{H_2}$)-free graph such  that $\overline{G}$ has no universal vertex or a pair of comparable vertices. Suppose that $G$ contains an $\overline{H_3}$ an $H_1$  with vertices and edges as in Figure~\ref{fig:partfig2}. Let $C:=\{v_1,v_2,v_3,v_4\}$ and we partition $V(G)\sm C$ as in Section~\ref{genprop}. We split the proof into two cases depending on the edges with one end in $A_i$ and the other in $A_{i+2}$, where $i\in \{1,2,3,4\}$, and the theorem follows from \cref{H3-AiAi+2ancomforali} and \cref{H3-AiAi+2notanticom}.
 \end{proof}

\subsection{($P_2+P_3$, $\overline{P_2+P_3}$,  $K_{2,3}$, banner,  $H_3$)-free graphs that contain a $C_4$}

Let $G$ be a ($P_2+P_3$, $\overline{P_2+P_3}$, $K_{2,3}$, banner, $H_2$, $H_3$)-free graph which has   no universal  or a pair of comparable vertices. Suppose that $G$ contains a $C_4$, say with vertex-set $C:=\{v_1,v_2,v_3,v_4\}$ and edge-set $\{v_1v_2,v_2v_3,v_3v_4,v_4v_1\}$.   We partition $V(G)\sm C$ as in Section~\ref{genprop},  and we use the properties in Section~\ref{genprop}. Clearly, since $G$ is $K_{2,3}$-free, $X=\es$, and since $G$ is banner-free, $A=\es$.  Moreover, the graph $G$ has some more properties which we give in \cref{final-Prop,lem:BiBi+2emp,lem:BiBi+2com} below.

 \begin{lemma}\label{final-Prop}For  $i\in \{1,2,3,4\}$, the following hold:

\vspace{-0.3cm}
\begin{lemmalist}
\item\label{BiBi+1comp}    $B_i$ is complete to $B_{i+1}\cup B_{i-1}$.
\item\label{Biclq}  If $B_{i+1}\neq \es$, then $B_i$ is a clique.
\end{lemmalist}
\end{lemma}
\begin{proof}
$(i)$:~If there are nonadjacent vertices, say $b\in B_i$ and $b'\in B_{i+1}\cup B_{i-1}$, then $C\cup \{b,b'\}$ induces an $H_3$. $\diamond$

\no{$(ii)$}:~If there are nonadjacent vertices, say $b,b'\in B_i$, then, for any $b'' \in B_{i+1}$, by \cref{BiBi+1comp}, $\{b,v_i,b',$ $b'',v_{i-1}\}$ induces a banner.
\end{proof}

 \begin{lemma}\label{lem:BiBi+2emp}
	If  $B_{i}$ and $B_{i+2}$ are empty, for some $i\in \{1,2,3,4\}$, then  $G$ is a good graph.
\end{lemma}
\begin{proof} We may assume that $B_2\cup B_4=\es$. If $B_1$ is anticomplete to $B_3$, then we define $S_1:=\{v_1,v_3\}$, $S_2:=\{v_2\}$ and $S_3:=T\cup \{v_4\}$. Then clearly, $S_1, S_2$ and $S_3$ are stable sets such that $\omega(G-(S_1\cup S_2\cup S_3))\leq \omega(G)-2$, and so $G$ is a nice graph. So, we may assume that, there are vertices, say $b_1\in B_1$ and $b_3\in B_3$ such that $b_1b_3\in E(G)$. Suppose that there are nonadjacent vertices, say $b,b'\in B_1$. If $b,b'\neq b_1$, then by \ref{prop-K23}:\ref{a1b1b3nbd},  $\{b_3,v_4,b,v_2,b'\}$ induces a $P_2+P_3$, and if $b=b_1$, then $\{b_1,b_3,v_4,v_1,b'\}$ induces a banner or a $K_{2,3}$. So, we conclude that $B_1$ is a clique. Likewise, $B_3$ is  a clique.
Now, if $D= \es$, then, by \ref{prop-K23}:\ref{a1b1b3nbd}, $G[B_1\cup B_3\cup C]$ is the  complement of a bipartite graph, and then, by \ref{ATstset}, we have $\chi(G)\leq \omega(G)+3$; so $D\neq \es$. If there is a vertex, say $d\in D$ such that $\{d\}$ is complete to either $B_1$ or $B_3$, then we let $S_1:=\{d\}$, otherwise let $S_1$ be the maximum stable set in $G[B_1\cup B_3\cup D]$ such that $D\cap S_1\neq \es$ and $(B_1\cup B_3)\cap S_1\neq \es$. Let $S_2:=T\cup \{v_1,v_3\}$ and $S_3:=\{v_2,v_4\}$. Then $S_1, S_2$ and $S_3$ are stable sets. We claim that for any maximum clique $Q$ in $G-(S_1\cup S_2\cup S_3)$,  we have $|Q|\leq \omega(G)-2$.  Suppose not, and let $K$ be a maximum clique  in $G-(S_1\cup S_2\cup S_3)$ such that $|K|\geq \omega(G)-1$.
Since $\{v_1,v_2\}$ is complete to $B_1\cup D$, and $\{v_3,v_4\}$ is complete to $B_3\cup D$, we may  assume that, $K\cap B_1, K\cap B_3\neq \es$. By \ref{prop-K23}:\ref{a1b1b3nbd}, we let $K\cap B_1:=\{b_1'\}$ and $K\cap B_3:=\{b_3'\}$.
If $D\cap S_1$ is not anticomplete to $K$, then, by \ref{bibi+2cmnnbd}, for any $d\in (D\cap S_1)\cap N(b_1')$, $(K\sm \{b_3'\})\cup \{d,v_1,v_2\}$ is a clique of size at least $\omega(G)+1$, a contradiction; so $D\cap S_1$ is anticomplete to $K$. Moreover, if $B_1\cap S_1=\es$, then by \ref{prop-K23}:\ref{a1b1b3nbd}, $S_1\cup \{b_1'\}$ is a stable set which  contradicts the choice of $S_1$; so $B_1\cap S_1\neq \es$. Likewise, $B_3\cap S_1\neq \es$. Then for any $b_1\in B_1\cap S_1$, $b_3\in B_3\cap S_1$ and $d\in D\cap S_1$, by \ref{prop-K23}:\ref{a1b1b3nbd}, $\{b_3,b_3',b_1,v_1,d\}$ induces a $P_2+P_3$, a contradiction. So $G$ is a nice graph, and that $G$ is a good graph.
  \end{proof}

\begin{lemma}\label{lem:BiBi+2com}
	If $B_{i}$ is complete to $B_{i+2}$, for each $i\in \{1,2,3,4\}$,  then  $G$ is a good graph.
\end{lemma}
\begin{proof}
 Since $N(v_i)=B_i\cup B_{i-1}\cup D\cup \{v_{i+1},v_{i-1}\}$, $v_i$ and $v_{i+2}$ are not comparable vertices, we may assume that $B_1$ and $B_3$ are nonempty. Then, by \ref{prop-K23}:\ref{a1b1b3nbd}, $|B_1|\leq 1$ and $|B_3|\leq 1$.  By \cref{lem:BiBi+2emp}, we may assume that $B_2\neq \es$. So by \cref{Biclq}, we may assume that $B_i$ is a clique, for each $i$. First suppose that $B_4\neq \es$. So, again by \ref{prop-K23}:\ref{a1b1b3nbd}, $|B_i|\leq 1$, for each $i$. For $i\in \{1,2,3,4\}$, we define $W_i:=B_{i}\cup \{v_{i+2}\}$ and $W_5:=T$. Then clearly $W_i$'s are stable sets, and hence, by \ref{Dper},  we have $\chi(G)\leq \omega(G)+3$. So we may assume that $B_4=\es$, and we define  $S_1:=B_1\cup \{v_3\}$, $S_2:=B_3\cup \{v_1\}$ and $S_3:=T\cup \{v_2,v_4\}$. Then $S_1, S_2$ and $S_3$ are three stable sets. Now if there  is a clique, say $Q\in G-(S_1\cup S_2\cup S_3)$ such that $|Q|>\omega(G)-2$, then  $Q\cap (B_2\cup D)\neq \es$, and then $Q\cup \{v_2,v_3\}$ is a clique of size $\omega(G)+1$, a contradiction. So $\omega(G-(S_1\cup S_2\cup S_3)) \leq \omega(G)-2$, and hence $G$ is a nice graph. Thus $G$ is a good graph.
\end{proof}

\begin{thm}\label{thm:c4}
	If $G$ is a ($P_2+P_3$, $\overline{P_2+P_3}$, $K_{2,3}$)-free graph that contains a $C_4$, then $G$ is a good graph.
\end{thm}
\begin{proof}
Let $G$ be a ($P_2+P_3$, $\overline{P_2+P_3}$, $K_{2,3}$)-free graph. Suppose that $G$ contains a $C_4$, say with vertex-set $C:=\{v_1,v_2,v_3,v_4\}$ and edge-set $\{v_1v_2,v_2v_3,v_3v_4,v_4v_1\}$.    We partition $V(G)\sm C$ as in Section~\ref{genprop},  and we use the properties in Section~\ref{genprop}. By Theorems~\ref{thm:banner}, \ref{thm:5apple} and \ref{thm:caseH3}, we may assume that $G$ is (banner, $H_2,H_3$)-free.  As earlier, since $G$ is $K_{2,3}$-free, $X=\es$, and since $G$ is banner-free, $A=\es$. We may assume that $G$ has   no universal   or a pair of comparable vertices, and we use the  \cref{final-Prop,lem:BiBi+2emp,lem:BiBi+2com}. Recall that, by \cref{BiBi+1comp}, for each $i\in \{1,2,3,4\}$, $B_i$ is complete to $B_{i+1}$. By \cref{lem:BiBi+2emp,lem:BiBi+2com}, we may assume that there are vertices $b_1\in B_1$, $b_2\in B_2$ and $b_3\in B_3$ such that $b_1b_2,b_2b_3\in E(G)$ and $b_1b_3\notin E(G)$. Then, by \cref{Biclq}, for each $i\in \{1,2,3,4\}$, $B_i$ is a clique.
If $D=\es$, then since $B_1\cup B_2\cup \{v_2\}$ and $B_3\cup B_4\cup \{v_4\}$ are cliques, $B\cup \{v_2,v_4\}$ induces a complement of a bipartite graph, and hence $\chi(G)\leq \chi(G[B\cup \{v_2,v_4\}])+\chi(T\cup \{v_1,v_3\}) \leq \omega(G)+1$. So we may assume that $D\neq \es$.
Now we claim the following:

\begin{claim}\label{C4b1bdcmnbdB2com}
	Any vertex in $D$ which is complete to $\{b_1,b_3\}$, is complete to $B_2\cup B_4$. Also, any vertex in $D$ which is anticomplete to $\{b_1,b_3\}$, is anticomplete to $B_2\cup B_4$.
\end{claim}

\vspace{-0.2cm}
\no{\it Proof of \cref{C4b1bdcmnbdB2com}}.~Let $d\in D$. If $db_1,db_3\in E(G)$, and there  is a vertex $b\in B_2$ (up to symmetry) such that $db\notin E(G)$, then $\{b_1,b,v_3,d,b_3\}$ induces a $\overline{P_2+P_3}$. If $db_1,db_3\notin E(G)$, and there  is a vertex $b'\in B_2$ (up to symmetry) such that $db'\in E(G)$, then $\{b_1,b',d,v_1,b_3\}$ induces a banner. $\Diamond$

Let $B':=\{b_1,b_2,b_3\}$.   To proceed further, we let:

\begin{tabular}{ll}
~~~$D_1:=\{d\in D\mid N(d)\cap B'=\{b_1\}\}$, &~~~$ D_1':=\{d\in D\mid N(d)\cap B'=\{b_3\}\}$,\\
~~~$D_2:=\{d\in D\mid N(d)\cap B'=\{b_1,b_2\}\}$, &~~~$ D_2':=\{d\in D\mid N(d)\cap B'=\{b_2,b_3\}\}$,\\
~~~$D_3:=\{d\in D\mid N(d)\cap B'=B'\}$, ~~ and  &~~~$ D_3':=\{d\in D\mid N(d)\cap B'=\es\}$.
   \end{tabular}

\no Then, by \cref{C4b1bdcmnbdB2com},  $D=\cup_{j=1}^3(D_j\cup D_j')$, and, by \ref{abtnbBD}, $D_1,D_1',D_2,D_2'$ and $D_3$ are cliques. Moreover:

\begin{claim}\label{D1D2stset}
 $D_1\cup D_2'$ is a stable set. Likewise, $D_2\cup  D_1'$ is a stable set.	
\end{claim}

\vspace{-0.2cm}
\no{\it Proof of \cref{D1D2stset}}.~If there are vertices, say $d,d'\in D_1$, then $\{b_1,b_2,v_3,d,d'\}$ induces a $\overline{P_2+P_3}$; so $|D_1|\leq 1$. By using a similar argument, we have $|D_2'|\leq 1$. If there are adjacent vertices, say $d\in D_1$ and $d'\in D_2'$, then $\{b_1,d,d',b_2,v_1\}$ induces a $\overline{P_2+P_3}$. So $D_1\cup D_2'$ is a stable set. Likewise, $D_2\cup  D_1'$ is also a  stable set. $\Diamond$

\begin{claim}\label{D4card}
	$| D_3'|\leq 1$.
\end{claim}

\vspace{-0.2cm}
\no{\it Proof of \cref{D4card}}.~If there are vertices, say $d_1,d_2\in  D_3'$, then, since $\{d_1,d_2,b_1,b_2,b_3\}$ does not induce a $P_2+P_3$, $d_1d_2\notin E(G)$, and then $\{b_1,b_2,d_1,v_4,d_2\}$ induces a $P_2+P_3$.   $\Diamond$

Now we prove the theorem in three cases as follows:

\smallskip
\no{\bf Case~1}~ {\it Suppose that   $D_1\cup  D_1'\neq \es$.}

 \smallskip
Up to symmetry, we may assume that there is a vertex, say $d\in D_1$. Then:

\begin{claim}\label{C4B2B4card}
	$|B_2|\leq 2$ and $|B_4|\leq 2$.
\end{claim}

\vspace{-0.2cm}
\no{\it Proof of \cref{C4B2B4card}}.~If there are vertices, say $b,b'\in B_2\sm \{b_2\}$, then since $\{b,b_3,v_4,d,b'\}$ does not induce a $\overline{P_2+P_3}$, we may assume that $bd\notin E(G)$, and then $\{b_1,b_2,v_3,d,b\}$ induces a $\overline{P_2+P_3}$; so $|B_2|\leq 2$. If there are vertices, say $p,q,r\in B_4$, then, since $\{b_3,v_3,d,p,q\}$ does not induce a  $\overline{P_2+P_3}$, we may assume that $pd\notin E(G)$, and then we get a contradiction as in the proof for $|B_2|\leq 2$.   $\Diamond$

So, by \ref{prop-K23}:\ref{a1b1b3nbd} and \cref{C4B2B4card}, we have  $\chi(G[B_2\cup B_4])\leq 2$.

\smallskip
\no{\bf Case~1.1}~ {\it Suppose that $B_1\sm\{b_1\}\neq \es$.}

Now we have the following:
\begin{claim}\label{C4B1B3D1chrom}
	$\chi(G[B_1\cup B_3\cup D_3\cup  D_3'])\leq \omega(G)-|B_2|-1$. Likewise, $\chi(G[B_1\cup B_3\cup D_3\cup D_3'])\leq \omega(G)-|B_4|-1$.
\end{claim}

\vspace{-0.2cm}
\no{\it Proof of \cref{C4B1B3D1chrom}}.~Since for any $b\in (B_1\sm\{b_1\})\cup (B_3\sm N(b_1))$, $\{b_1,b_2,v_3,d,b\}$ does not induce a  $\overline{P_2+P_3}$, $\{d\}$ is anticomplete to  $(B_1\sm\{b_1\})\cup (B_3\sm N(b_1))$. Thus, if there are nonadjacent vertices, say $b'\in B_1\sm\{b_1\}$ and $b''\in B_3\sm N(b_1)$, then $\{b_1,b_2,v_3,d,b',b''\}$ induces an $\overline{H_3}$; so $B_1\sm \{b_1\}$ is complete to $B_3\sm N(b_1)$.
Then since $b_3\in B_3\sm N(b_1)$, by \ref{prop-K23}:\ref{a1b1b3nbd}, $|B_1\sm \{b_1\}|=1$ (so $|B_1|=2$), and $B_3\sm N(b_1)=\{b_3\}$.

Hence, by \ref{bibi+2cmnnbd}, $D_3$ is complete to $B_1\sm \{b_1\}$. So, by \cref{C4b1bdcmnbdB2com}, $D_3\cup B_1\cup B_2\cup \{v_2\}$ is a clique. By  \cref{D4card}, $ D_3'\cup \{b_1,b_3\}$  is a stable set.  By \ref{prop-K23}:\ref{a1b1b3nbd}, $(B_1\sm \{b_1\})\cup (B_3\sm \{b_3\})$ is a stable set.
   So, $\chi(G[B_1\cup B_3\cup D_3\cup D_3'])\leq \omega(G[D_3])+\chi(G[ D_3'\cup \{b_1,b_3\}])+ \chi(G[(B_1\sm \{b_1\})\cup (B_3\sm \{b_3\})]) =(\omega(G[D_3\cup B_1\cup B_2\cup \{v_2\}])-|B_1|-|B_2|-1)+2\leq \omega(G)-2-|B_2|+1=\omega(G)-|B_2|-1$.  Likewise, $\chi(G[B_1\cup B_3\cup D_3])\leq \omega(G)-|B_4|-1$. This proves \cref{C4B1B3D1chrom}. $\Diamond$

   By \cref{D1D2stset},
$\chi(G[D_1\cup  D_1'\cup D_2\cup D_2'])\leq 2$. So by \cref{C4B1B3D1chrom}, $\chi(G[B_1\cup B_3\cup D])\leq \omega(G)-|B_2|+1$ and $\chi(G[B_1\cup B_3\cup D])\leq \omega(G)-|B_4|+1$.
Now if $|B_2|= 1$ and $|B_4|\leq 1$, then, by \ref{ATstset}, $B_2\cup\{v_1\}$, $B_4\cup \{v_3\}$ and $T\cup\{v_2,v_4\}$ are stable sets, and hence we conclude that $\chi(G)\leq \omega(G)+3$. So by \cref{C4B2B4card}, we may assume that either  $|B_2|=2$ or $|B_4|=2$. Then $\chi(G[B_1\cup B_3\cup D])\leq \omega(G)-1$. Recall that $\chi(G[B_2\cup B_4])\leq 2$. Since $T\cup \{v_2,v_4\}$ and $\{v_1,v_3\}$ are stable sets (by \ref{ATstset}), we get $\chi(G)\leq \omega(G)+3$.

\smallskip
\no{\bf Case~1.2}~ {\it Suppose that $B_1\sm \{b_1\}=\es$.}

\smallskip
If $B_3\sm \{b_3\}=\es$, then since $\{b_1,v_3\}$, $\{b_3,v_1\}$ and $T\cup \{v_2,v_4\}$ are stable sets (by \ref{ATstset}), and since $\chi(G[B_2\cup B_4])\leq 2$,  by \ref{Dper},  $\chi(G)\leq \chi(G-D)+\chi(G[D])\leq 5+(\omega(G)-2)=\omega(G)+3$. So we assume that $B_3\sm\{b_3\}\neq \es$. Then, we may assume that $D_1'=\es$. For, otherwise, if there is a vertex $d_1\in D_1'$, then $b_3d_1\in E(G)$, $b_1d_1,b_2d_1\notin E(G)$ and  $B_3\sm \{b_3\}\neq \es$, and thus this case is similar to that of Case~1.1. 	
	Also:

 \begin{claim}\label{B4card}
 We may assume that $|B_4|\leq 1$.
 \end{claim}

   \vspace{-0.2cm}
\no{\it Proof of \cref{B4card}}.~If there are vertices, say $b_4,b_4'\in B_4$, then by \ref{prop-K23}:\ref{a1b1b3nbd}, we may assume $b_2b_4\notin E(G)$, then since $\{b_1,b_2,b_3,b_4,d\}$ does not induce a banner, $b_4d\in E(G)$. Thus, we conclude that $b_3d\notin E(G)$, $b_2d\notin E(G)$, $b_2b_4\notin E(G)$, and $B_4\sm \{b_4\}\neq \es$, and again the case is similar to that of Case~1.1.  $\Diamond$

By \ref{ATstset} and \cref{B4card},  $\{b_1,v_4\}$, $B_4\cup \{v_2\}$ and $T\cup \{v_1,v_3\}$ are stable sets.  By \cref{D1D2stset},  $D_1\cup \{b_2\}$ is a stable set. By \cref{C4b1bdcmnbdB2com,D4card,C4B2B4card}, $(B_2\sm \{b_2\})\cup  D_3'$ is a stable set. Hence $\chi(G[B_1\cup B_2\cup B_4\cup D_1\cup  D_3'])\leq 5$.
Now $D\sm (D_1\cup  D_3')$ is complete to $\{b_2\}$. Then, by \ref{abtnbBD}, $D\sm (D_1\cup  D_3')$ is a clique. Thus $G[B_3\cup (D\sm (D_1\cup  D_3'))]$ is the complement of a bipartite graph, and hence a perfect graph. Since $\{v_3,v_4\}$ is complete to $B_3\cup (D\sm (D_1\cup  D_3'))$, $\chi(G[B_3\cup (D\sm (D_1\cup D_3'))])\leq \omega(G)-2$.
Hence $\chi(G)\leq \omega(G)+3$.

\smallskip
\no{\bf Case~2}~ {\it Suppose that $D_1\cup D_1'=\es$  and $D_2\cup D_2'\neq \es$.}

Up to symmetry, we may assume that there is a vertex, say $d\in D_2$. Then $b_1d,b_2d\in E(G)$, and $b_3d\notin E(G)$. Now we have the following:

\begin{claim}\label{B2b2}
We may assume that $B_2=\{b_2\}$.
\end{claim}

\vspace{-0.2cm}
\no{\it Proof of \cref{B2b2}}.~If there is a vertex, say $b_2'\in B_2\sm\{b_2\}$, then, since $\{b_2,b_3,v_4,d,b_2'\}$ does not induce a  $\overline{P_2+P_3}$, $b_2'd\notin E(G)$. Now, we see that there is a vertex  $d\in D$ such that $b_1d\in E(G)$, $b_2'd,b_3d\notin E(G)$, and the case is similar to that of Case~1.  $\Diamond$

\begin{claim}\label{B4compB2}We may assume $B_4$ is complete to $\{b_2\}$.\end{claim}
	
\vspace{-0.2cm}
\no{\it Proof of \cref{B4compB2}}.~Suppose there is a vertex, say $b_4\in B_4$ such that $b_2b_4\notin E(G)$.  Then since $\{b_1,b_2,b_3,b_4,d\}$ does not induce a  $\overline{P_2+P_3}$,  $b_4d\notin E(G)$.
 Now, we see that there is a vertex  $d\in D$ such that $b_2d\in E(G)$, $b_3d,b_4d\notin E(G)$, and the case is similar to that of Case~1.  $\Diamond$

 \begin{claim}\label{B1B3antcomD3}We may assume that $B_1\cup B_3$ is anticomplete to $ D_3'$.\end{claim}
	
\vspace{-0.2cm}
\no{\it Proof of \cref{B1B3antcomD3}}.~Suppose there are adjacent vertices, say $b\in B_1$ and $d\in  D_3'$. Then, by \ref{bibi+2cmnnbd}, $bb_3\notin E(G)$. Now, we see that there is a vertex  $d\in D$ such that $bd\in E(G)$, $b_2d,b_3d\notin E(G)$, and the case is similar to that of Case~1. So $B_1$ is anticomplete to $ D_3'$. Likewise, $B_3$ is anticomplete to $ D_3'$. $\Diamond$

\begin{claim}\label{dcomB1B3}
	$\{d\}$ is complete to $(B_1\sm N(b_3))\cup (B_3\sm \{b_3\})\cup (D\sm \{d\})$.
\end{claim}

\vspace{-.2cm}
\no{\it Proof of \cref{dcomB1B3}}.~Suppose there is a vertex, say $p\in (B_1\sm N(b_3))\cup (B_3\sm \{b_3\})\cup (D\sm \{d\})$ such that $pd\notin E(G)$. If $p\in B_1\sm N(b_3)$, then $\{p,b_2,d,v_1,b_3\}$ induces a banner.
If  $p\in B_3\sm \{b_3\}$,  then $\{b_2,b_3,v_4,d,p\}$ induces a  $\overline{P_2+P_3}$. Since $\{b_2\}$ is complete to $D\sm  D_3'$, by \ref{abtnbBD}, $D\sm  D_3'$ is a clique; so $\{d\}$  is complete to $D\sm  D_3'$. So, we see that $p\in D_3'$. But then $\{d,v_1,p,v_3,b_3\}$ induces a banner.
This proves \cref{dcomB1B3}. $\Diamond$

First suppose that $B_4\neq \es$. By \ref{prop-K23}:\ref{a1b1b3nbd} and \cref{B4compB2}, $|B_4|=1$. By \ref{bibi+2cmnnbd} and \cref{B4compB2}, $B_4$ is complete to $D\sm  D_3'$. Now we define three stable sets, namely $S_1:=\{b_2,v_4\}$, $S_2:=B_4\cup \{v_2\}$ and $S_3:=T\cup \{v_1,v_3\}$. We claim that for any maximum clique $Q$ in $G- (S_1\cup S_2\cup S_3)$,  $|Q|\leq \omega(G)-2$. Suppose not, and let $K$ be a maximum clique  in $G- (S_1\cup S_2\cup S_3)$ such that $|K|\geq \omega(G)-1$. Since  for $j\in\{1,3\}$, $\{v_j,v_{j+1}\}$ is complete to $B_j\cup D$, we may assume that $K\cap B_1\neq \es$ and $K\cap B_3\neq \es$.  Then, by \cref{B1B3antcomD3}, $K\cap D_3'=\es$. Since $B_4\cup \{b_2\}$ is complete to $B_1\cup B_3\cup (D\sm  D_3')$, we see that $K\cup B_4\cup \{b_2\}$ is a clique of size $\omega(G)+1$, a contradiction. Hence $G$ is nice. So $\omega(G- (S_1\cup S_2\cup S_3)\leq \omega(G)-2$, and hence $G$ is a nice graph.
So we may assume that $B_4=\es$. Now we define three stable sets $S_1:=\{b_3,d\}$, $S_2:=\{b_2,v_1\}$, and $S_3:=T\cup \{v_2,v_4\}$, and we claim the following:

\begin{claim}\label{C4p2B2Dcom}
	For any maximum clique $Q$ in $G-(S_1\cup S_2\cup S_3)$, we have $|Q|\leq \omega(G)-2$.
\end{claim}

\vspace{-0.2cm}
\no{\it Proof of \cref{C4p2B2Dcom}}.~Suppose not, and let $K$ be a maximum clique in $G- (S_1\cup S_2\cup S_3)$ such that  $|K|\geq \omega(G)-1$. If   $K\cap B_1=\es$, then, by \cref{dcomB1B3}, $K\cup \{v_4,d\}$ is a clique of size at least $\omega(G)+1$, a contradiction; so $K\cap B_1\neq \es$. If $K\cap B_3\neq \es$, then $K\cap B_3$ is complete to $\{d\}$ (by \cref{dcomB1B3}) and $K\cap B_1$ is complete to $\{d\}$ (by \ref{bibi+2cmnnbd}), and then $K\cup\{b_2,d\}$ is a clique of size at least $\omega(G)+1$, a contradiction; so $K\cap B_3=\es$. Then $|K\cup \{v_1,v_2\}|\geq \omega(G)+1$, a contradiction.
	This proves \cref{C4p2B2Dcom}. $\Diamond$

Hence, by \cref{C4p2B2Dcom},  $G$ is a good graph.

\smallskip
\no{\bf Case~3}~ {\it Suppose that $D_1\cup D_1' \cup D_2\cup D_2'=\es$.}

 \smallskip
Suppose that there are nonadjacent vertices, say $b\in B_1$ and $d\in D_3$, then, by \ref{prop-K23}:\ref{a1b1b3nbd}, $bb_3\notin E(G)$. Now, we see that there is a vertex  $d\in D$ such that $b_2d,b_3d\in E(G)$, $bd\notin E(G)$, and the case is similar to that of Case~2.  So, $B_1$ is complete to $D_3$. Likewise, $B_3$ is complete to $D_3$.  Since $B_1\cup B_2$ and $B_3\cup B_4$ are cliques, $G[B]$ induces the complement of a bipartite graph.  Also, by \cref{C4b1bdcmnbdB2com}, $D_3$ is complete to $B$. So, $G[B\cup D_3]$ induces a perfect graph. Since, by \ref{ATstset} and \cref{D4card}, $ D_3'$, $\{v_2,v_4\}$ and $T\cup \{v_1,v_3\}$ are stable sets, we conclude that $\chi(G)\leq \omega(G)+3$. \end{proof}

\medskip
\no{\bf Proof of \cref{thm:structure}}. The proof  follows from  Theorem~\ref{thm:k23}  and Theorem~\ref{thm:c4}. \hfill{$\Box$}

\section{Proof of \cref{thm:p2p3-bnd}}\label{sec:col}

In this section, we give a proof of \cref{thm:p2p3-bnd}. We will use the following lemmas.

 \begin{lemma}\label{lem:pec}
  If $G$ is a nice graph, then  $\chi(G)\leq \lfloor\frac{3}{2} \omega(G) \rfloor-1$.
 \end{lemma}
  \begin{proof}  Let $G$ be a nice graph. Then $G$ has three pairwise disjoint stable sets, say $S_1,S_2$ and $S_3$, such that $\omega(G-(S_1\cup S_2\cup S_3))\leq \omega(G)-2$. Let $S:=S_1\cup S_2\cup S_3$.  We prove the lemma by induction on $|V(G)|$.   Since $\chi(G)\leq \chi(G-S)+\chi(G[S])$, by induction hypothesis, $\chi(G) \leq  (\lfloor\frac{3}{2}\omega(G-S) \rfloor-1)+3 \leq  (\lfloor\frac{3}{2} (\omega(G)-2) \rfloor-1) +3 \leq \lfloor\frac{3}{2} \omega(G) \rfloor-1$.
  \end{proof}

  \begin{lemma}\label{lem:col:nice-ver}
  If $G$ has a nice vertex, then $\chi(G)\leq \omega(G)+3$.
  \end{lemma}
  \begin{proof}
  Suppose that $G$ has a nice vertex, say $u$. We prove the lemma by induction on $|V(G)|$.  Now, since $d_G(u)\leq \omega(G)+2$, we can
take any $\chi(G)$-coloring of $G- u$ and extend it to a
$\chi(G)$-coloring of $G$, using for $u$ a
 color (possibly new) that does not appear in $N_G(u)$.
  \end{proof}

 \begin{lemma}\label{lem:col-nice}
If $G$ is a good graph, then $\chi(G)\leq \max \{\omega(G)+3, \lfloor\frac{3}{2} \omega(G) \rfloor-1 \}$.
\end{lemma}
\begin{proof}
Let $G$ be a good graph. If $\chi(G)\leq \omega(G)+3$ or  $G$ is a nice graph or if $G$ has a nice vertex, then, by \cref{lem:pec,lem:col:nice-ver},  we are done.   So we may assume that either $G$ has a universal vertex or $G$ has a pair of comparable vertices.   Now we prove the lemma by induction on $|V(G)|$.
If $G$ has a universal vertex, say $u$, then $\omega(G-u)=\omega(G)-1$, and then $\chi(G)= \chi(G-u)+1\leq \max \{\omega(G-u)+3, \lfloor\frac{3}{2} \omega(G-u) \rfloor-1 \}+1 \leq \max \{\omega(G)+3, \lfloor\frac{3}{2} \omega(G) \rfloor-1\}$, and we are done. Next, if $G$ has a pair of comparable vertices, say $u$ and $v$, such that $N(u)\subseteq N(v)$ (say), then $\chi(G) =\chi(G-u)$ and
$\omega(G) = \omega(G-u)$; so we can
take any $\chi(G)$-coloring of $G- u$ and extend it to a
$\chi(G)$-coloring of $G$, using for $u$ the color of $v$, and we conclude the proof.
\end{proof}

  We will also use the following result.
\begin{thm}[\cite{KMa18}]\label{thm:p2p3c4}
If $G$ is a ($P_6$, $C_4$)-free graph, then $\chi(G)\leq  \lceil\frac{5}{4} \omega(G) \rceil$. Moreover, the bound is tight.
\end{thm}

\smallskip

\no{\bf Proof of \cref{thm:p2p3-bnd}}. Let $G$ be a ($P_2+ P_3$, $\overline{P_2+ P_3}$)-free graph. By \cref{thm:p2p3c4}, since $P_2+P_3$ is an induced subgraph of $P_6$, we may assume that $G$ contains a $C_4$. Then, by \cref{thm:structure}, $G$ is a good graph, and hence the proof follows from  \cref{lem:col-nice}. $\hfill{\Box}$

\bigskip
\no{\bf Acknowledgement}. We thank Professor S.\,A.\,Choudum for his valuable suggestions.

\end{document}